\documentclass{article}

\usepackage{coge}

\title{Orientation of convex sets}

\author[1,2]{P\'eter \'Agoston}
\author[1,2]{G\'abor Dam\'asdi\thanks{Supported by the \'{U}NKP-21-3 New National Excellence Program of the Ministry for Innovation and Technology from the source of the National Research, Development and Innovation fund.}}
\author[1,2]{Bal\'azs Keszegh\thanks{Research supported by the Lend\"ulet program of the Hungarian Academy of Sciences (MTA), under the grant LP2017-19/2017, by the J\'anos Bolyai Research Scholarship of the Hungarian Academy of Sciences, by the National Research, Development and Innovation Office -- NKFIH under the grant K 132696 and FK 132060 and by the \'UNKP-20-5 New National Excellence Program of the Ministry for Innovation and Technology from the source of the National Research, Development and Innovation Fund.}}
\author[1,2]{D\"om\"ot\"or P\'alv\"olgyi}
\affil[1]{ELTE E\"{o}tv\"{o}s Lor\'{a}nd University, Budapest
}
\affil[2]{Alfr\'{e}d R\'{e}nyi Institute of Mathematics, Budapest
}
		
\date{}

\begin{document}

\maketitle

\begin{abstract}
    We introduce a novel definition of orientation on the triples of a family of pairwise intersecting planar convex sets and study its properties.
    In particular, we compare it to other systems of orientations on triples that satisfy a so-called interiority condition: $\o(ABD)=\o(BCD)=\o(CAD)=1$ imply $\o(ABC)=1$ for any $A, B, C, D$.
    We call such an orientation a \PO (partial 3-order), a natural generalization of a poset, that has several interesting special cases.
    For example, the order type of a planar point set (that can have collinear triples) is a \PO, but not every \PO is the order type of some planar point set; a \PO that is realizable by points is called a \pPO.
    
    If the family is non-degenerate with respect to the orientation, i.e., always $\o(ABC)\ne 0$, we obtain a \TO (total 3-order).
    Contrary to linear orders, a \TO can have a rich structure. A \TO realizable by points, a \pTO, is the order type of a point set in general position.
    Despite these similarities to order types, \PO's and \TO's that can arise from the orientation of pairwise intersecting convex sets, denoted by \CPO and \CTO, turn out to be quite different from order types: there is no containment relation among the family of all \CPO's and the family of all \pPO's, or among the families of \CTO's and \pTO's.
    
    Finally, we study properties of these orientations if we also require that the family of the underlying convex sets satisfies the (4,3) property, as a first step towards obtaining better $(p,q)$-theorems.
\end{abstract}

\section{Introduction}\label{sec:introduction}

A family is \emph{intersecting} if any two members of the family intersect, and it is \emph{3-intersection-free} if no three members of the family have a common intersection. 
In order to better understand the intersection structure of planar convex sets, we will define an orientation of such sets (for a more detailed motivation, see the beginning of Chapter \ref{app:43}).
First, we need to define a few other notions.

Take a family of objects, $\mathcal{F}$, a fixed positive integer $k$, and a function $\o$, whose domain is the set of $k$-tuples of distinct members of the family and its range is the set $\left\lbrace-1,0,1\right\rbrace$.
If $\o$ furthermore satisfies the condition $\o\left(A_{\sigma(1)},...,A_{\sigma(k)}\right)=sgn(\sigma)\cdot \o\left(A_1,...,A_k\right)$ for any $A_1,...,A_k\in\mathcal{F}$ and any permutation $\sigma$ of $\{1,\dots,k\}$, then we call such an assignment a \emph{partial orientation of rank $k$}. In the special case, when $\o$ nowhere vanishes, i.e., it is never $0$, we call such an assignment a \emph{total orientation}. Such a relation is called an \emph{alternating sign map} in matroid theory, see for example \cite[Definition 3.5.3]{OM}). In Knuth \cite{Knuth}, these are called \emph{cyclic symmetry} and \emph{antisymmetry} axioms in the case of the orientation of three planar points.
We prefer the term orientation and, just like it is common in the case of standard orderings, we sometimes omit writing `total' or `partial' before it. We will also use the term \emph{orientation} of the ordered $k$-tuple $\left(A_1,...,A_k\right)$ for the value $\o\left(A_1,...,A_k\right)$, which we will oftentimes only denote by $\o\left(A_1...A_k\right)$.

Suppose that in an orientation of rank $k$ on the family $\mathcal{F}$, an ordered $(k+1)$-tuple $\left(A_1,A_2,...,A_k,B\right)$ from $\mathcal{F}$ has the following property:
There is a $\varepsilon\in\lbrace-1,+1\rbrace$ (that can depend on the $(k+1)$-tuple) such that
for all permutations $\sigma$ of $\left\lbrace1,...,k\right\rbrace$, $\o\left(A_{\sigma(1)},A_{\sigma(2)},...,A_{\sigma(k-2)},A_{\sigma(k-1)},B\right)=\varepsilon\cdot sgn(\sigma)$. If for every ordered $(k+1)$-tuple having the property described above, $\o\left(A_1,...,A_k\right)=\varepsilon$ also holds (since $\o$ is an orientation, this is equivalent to that $\o\left(A_{\sigma(1)},A_{\sigma(2)},...,A_{\sigma(k-2)},A_{\sigma(k-1)},A_{\sigma(k)}\right)=\varepsilon\cdot sgn(\sigma)$ for all permutations $\sigma$), we say that the orientation satisfies the interiority condition (following Knuth, who defined the interiority condition for $k=3$) and also that it is an \emph{order of rank $k$} (or simply a $k$-order). If it is a partial orientation (so in the general case), we call it a \emph{partial order of rank $k$} (partial $k$-order or P$k$O), while if it is a total orientation, we call it a \emph{total order of rank $k$} (total $k$-order of T$k$O). 

The above definitions can also be stated equivalently with a cyclic condition if we separate the cases when $k$ is odd and even (the equivalence follows from examining the possible permutations of the elements):

If $k$ is odd, an orientation of order $k$ on a family $\mathcal{F}$ fulfills the interiority condition if and only if for any $(k+1)$-tuple $\left(A_1,A_2,...,A_k,B\right)$, if $\o\left(A_1A_2...A_{k-1}B\right)=\o\left(A_2A_3...A_kB\right)=\o\left(A_3A_4...A_kA_1B\right)=...=\o\left(A_kA_1A_2...A_{k-2}B\right)=1$, then $\o\left(A_1A_2...A_k\right)=1$.

If $k$ is even, an orientation of order $k$ on a family $\mathcal{F}$ fulfills the interiority condition if and only if for any $(k+1)$-tuple $\left(A_1,A_2,...,A_k,B\right)$, if $\o\left(A_1A_2...A_{k-1}B\right)=-\o\left(A_2A_3...A_kB\right)=\o\left(A_3A_4...A_kA_1B\right)=...=-\o\left(A_kA_1A_2...A_{k-2}B\right)=1$, then $\o\left(A_1A_2...A_k\right)=1$.

For any ordered $(k+1)$-tuple $\left(A_1,...,A_k,B\right)$ satisfying the premise of the interiority condition for any $k$-order, we say that $B\in conv\left(A_1A_2...A_k\right)$.
It is easy to check that the definition of $conv$ is invariant to any permutation of the $A_i$ and that for the standard orientation of points in general position (described in the next paragraph) it coincides with the standard notion of $B\in conv\left(A_1A_2...A_k\right)$, however, it still does not necessarily satisfy all natural properties of convexity, as we will see it at the end of this section.

\smallskip
Note that a P2O is just an ordinary poset, since if $k=2$, then the interiority condition is equivalent to transitivity.
A T2O is just a total linear order.
A well-known example of an order of rank $d+1$ is the standard orientation of points in $\R^d$, given by the sign of the determinant of the $(d+1)\times(d+1)$ matrix whose rows are the coordinates of the points with an extra $1$ at the end. From now on, whenever we use $\o$ for points, like $\o(abc)$ for some $a,b,c\in\mathbb R^2$, we always mean the above defined orientation.
Also note that when using the notation for three sets, we will drop the commas from the notation.

Some specific P$k$O's, called chirotopes \cite{OM} and signotopes \cite{FW} have also been studied earlier.

Once an order of some rank has been fixed (or is implicitly understood) on a (typically finite) family $\mathcal F$, then the order induced by it is called the \emph{order type} of $\mathcal F$,
so the order type of a family of objects for us is just the underlying P$k$O.
Traditionally, this term has been used for point sets with the above defined $\o$; we use it in a more general sense, for any family (with a fixed order of an appropriate rank).
Note that in the literature order type is often only used for points in general position; we also allow the sign to be 0 in an order type, as in the more restrictive case we can emphasize that the order type is a T$k$O.
We will call two order types \emph{isomorphic} if one can be obtained from the other one using the permutation of the elements of its domain, and the flipping of the signs of all non-zero values assigned to the $k$-tuples ($+1$ to $-1$ and $-1$ to $+1$).
This latter operation, in the case of points in $\R^d$, corresponds to reflecting the point set to a hyperplane, so it is natural to identify two such order types, as was done in \cite{BFSchSchS}; all our studied P$k$O's will have a similar symmetry.
We also want to emphasize that two different types of objects with two differently defined implicit orders
can have the same order type; we only need that the signs of the orientations are the same under some permutation, possibly after a flip of signs.

In this work, we will focus on the rank $3$ case and assume $k=3$ for all considered orientations.
So we will call an order of rank $3$ a {(partial) $3$-order (P3O), and a {total $3$-order (T3O) if it does not take 0 anywhere. Recall that the interiority condition is equivalent to that if $\o(ABD)=\o(BCD)=\o(CAD)=1$, then $\o(ABC)=1$ holds for all ordered $4$-tuples from a $3$-order. 
Note that the orientation of a planar point triple is determined by the order in which the points follow each other on the boundary of their convex hull (triangle): it is $+1$ if the three points are in a counterclockwise order and $-1$ if they are in a clockwise order, while if the three points are collinear, it is $0$. 
A P3O that is the order type of a point set is denoted as \pPO, and such a T3O is a \pTO, so \pTO correspond to traditional point order types, with no collinearities.

\smallskip

Now we proceed to convex sets.
From now on, whenever we refer to a convex set, it is always assumed to be closed and planar, unless stated otherwise. 
Our goal is to define a P3O on pairwise intersecting plane convex sets.

Nerve complexes are a well-known notion in topology, introduced by Alexandrov \cite{Alexandrov} (originally for open sets). Such a complex belongs to any family $\mathcal{F}=\left\lbrace S_i\vert i\in I\right\rbrace$ of sets over a topological space. It is defined as an abstract simplicial complex in which each $S_i$ is represented by a vertex $v_i$ and any finite subset $\left\lbrace v_j\vert j\in J\right\rbrace$ ($\lvert J\rvert<\infty$, $J\subseteq I$) of them is a face if and only if all the elements of $\left\lbrace S_j\vert j\in J\right\rbrace$ have a common intersection. Several related theorems, called `nerve theorems' exist that show a connection between the topology of a family of sets and its nerve. Two well-known examples are Leray's nerve theorem \cite{Leray} and Borsuk's nerve theorem \cite{Borsuk}. A variant of the nerve theorem (see \cite[Theorem~4.2.2.]{Roll} and see also \cite[Theorem~3.9.]{BKRR} for a similar statement) states that the nerve belonging to a family of compact convex sets is homotopy equivalent to the union of said sets (note that a similar result for good covers, which are a generalization of families of convex sets, has been proven in \cite[Theorem~13.4.]{BT}). This shows that if three convex planar sets, $A$, $B$, $C$, form an intersecting and 3-intersection-free family, then $A\cup B\cup C$ is homotopy equivalent to $\mathbb S^1$, from which the Jordan curve theorem shows that $\R^2\setminus (A\cup B\cup C)$ has exactly one bounded component, called the \emph{hollow} of $ABC$, which we will denote by $\hollow(ABC)$ (see Figure \ref{fig:hollow}).
Lehel and T\'oth \cite{LT} have also shown that the convex hull of this hollow is a triangle with sides $a,b,c$, such that (apart from its endpoints) side $a$ is contained in $A\setminus(B\cup C)$, side $b$ in $B\setminus(A\cup C)$, and side $c$ in $C\setminus(A\cup B)$.
We may refer to the vertices of this triangle as the \emph{vertices of the hollow}, but note that since the hollow is open, its vertices are not a part of it, only of its closure.

The following lemma, which enables us to define the orientation of triples of pairwise intersecting convex sets, is a straightforward consequence of Lemma 1 in \cite{JKLPT}; this result inspired us to study such families in more detail and from a different perspective.

\begin{figure}[h]
	\centering
	\begin{tikzpicture}[line cap=round,line join=round,>=triangle 45,x=0.6cm,y=0.6cm]
	\clip(-2.8876996321438195,0.6977106489352843) rectangle (3.2664020944047905,6.710338772574716);
	\draw [rotate around={2.480632746197529:(0.18220833836858297,2.0508869860909624)},line width=1.pt] (0.18220833836858297,2.0508869860909624) ellipse (1.1792475838121659cm and 0.7235894076586996cm);
	\draw [rotate around={-55.86852534016869:(1.5970647868681078,4.28267436999626)},line width=1.pt] (1.5970647868681078,4.28267436999626) ellipse (1.3508296865320082cm and 0.6212549196030458cm);
	\draw [rotate around={63.57747534364856:(-1.2349837699919162,4.282674369996265)},line width=1.pt] (-1.2349837699919162,4.282674369996265) ellipse (1.3482265070678305cm and 0.618488709170679cm);
	\draw [line width=0.8pt,dash pattern=on 3pt off 3pt] (-0.7714031045025065,3.0800906302351887)-- (0.10599954967347891,5.143042881429993);
	\draw [line width=0.8pt,dash pattern=on 3pt off 3pt] (0.10599954967347891,5.143042881429993)-- (1.21979079672695,3.1034935523559617);
	\draw [line width=0.8pt,dash pattern=on 3pt off 3pt] (-0.7714031045025065,3.0800906302351887)-- (1.21979079672695,3.1034935523559617);
	\fill[line width=0pt,fill=black,fill opacity=0.1] (-0.7714031045025065,3.0800906302351887) -- (0.10599954967347891,5.143042881429993) -- (1.21979079672695,3.1034935523559617) -- cycle;
	\draw (-1.1546479965065672,4.6) node[anchor=north west] {$a$};
	\draw (0.825982444221721,4.6) node[anchor=north west] {$b$};
	\draw (-0.02285917323325971,3.0) node[anchor=north west] {$c$};
	\draw (-2.7,6.0) node[anchor=north west] {$A$};
	\draw (2.4,6.0) node[anchor=north west] {$B$};
	\draw (2.0,1.8) node[anchor=north west] {$C$};
	\end{tikzpicture}
	\begin{tikzpicture}[line cap=round,line join=round,>=triangle 45,x=0.6cm,y=0.6cm]
	\clip(-2.8876996321438195,0.6977106489352843) rectangle (3.2664020944047905,6.710338772574716);
	\draw [rotate around={2.480632746197529:(0.18220833836858297,2.0508869860909624)},line width=1.pt] (0.18220833836858297,2.0508869860909624) ellipse (1.1792475838121659cm and 0.7235894076586996cm);
	\draw [rotate around={-55.86852534016869:(1.5970647868681078,4.28267436999626)},line width=1.pt] (1.5970647868681078,4.28267436999626) ellipse (1.3508296865320082cm and 0.6212549196030458cm);
	\draw [rotate around={63.57747534364856:(-1.2349837699919162,4.282674369996265)},line width=1.pt] (-1.2349837699919162,4.282674369996265) ellipse (1.3482265070678305cm and 0.618488709170679cm);
	\draw [line width=0.8pt,dash pattern=on 3pt off 3pt] (-0.7714031045025065,3.0800906302351887)-- (0.10599954967347891,5.143042881429993);
	\draw [line width=0.8pt,dash pattern=on 3pt off 3pt] (0.10599954967347891,5.143042881429993)-- (1.21979079672695,3.1034935523559617);
	\draw [line width=0.8pt,dash pattern=on 3pt off 3pt] (-0.7714031045025065,3.0800906302351887)-- (1.21979079672695,3.1034935523559617);
	\fill[line width=0pt,fill=black,fill opacity=0.1] (-0.7714031045025065,3.0800906302351887) -- (0.10599954967347891,5.143042881429993) -- (1.21979079672695,3.1034935523559617) -- cycle;
	\draw (-1.1546479965065672,4.6) node[anchor=north west] {$a$};
	\draw (0.825982444221721,4.7) node[anchor=north west] {$c$};
	\draw (-0.02285917323325971,3.0) node[anchor=north west] {$b$};
	\draw (-2.7,6.0) node[anchor=north west] {$A$};
	\draw (2.4,6.0) node[anchor=north west] {$C$};
	\draw (2.0,1.8) node[anchor=north west] {$B$};
	\end{tikzpicture}
	\caption{Three convex sets, $A$, $B$, $C$, with negative orientation on the left, and with positive orientation on the right, and their hollow,\protect\hollow\!\!($ABC$).}\label{fig:hollow}
\end{figure}
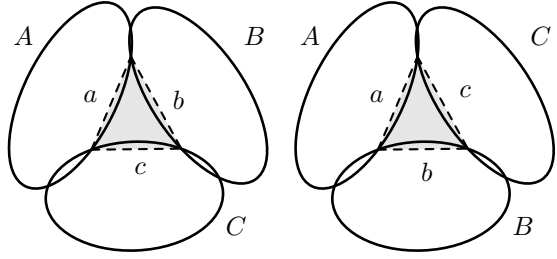

\begin{lem}[Jobson-K\'ezdy-Lehel-Pervenecki-T\'oth \cite{JKLPT}]\label{lem:hollow}
	Three pairwise intersecting closed convex sets, $A,B,C$, that do not have a common point, enclose a hollow $\hollow(ABC)$, and the following four properties hold.
	
	\begin{enumerate}[label=(\alph*)]
		\item $\hollow(ABC)$ is a simply connected region.
		\item The boundary of $\hollow(ABC)$ has exactly one arc from each of the boundaries of $A$, $B$ and $C$.
		\item The closure of the convex hull of $\hollow(ABC)$ is a triangle with sides $a,b,c$ such that (apart from its endpoints) side $a$ is contained in $A\setminus(B\cup C)$, side $b$ in $B\setminus(A\cup C)$, and side $c$ in $C\setminus(A\cup B)$
		\item For any $x\in B\cap C$, $y\in A\cap C$ and $z\in A\cap B$ the orientation of $x,y,z$ is the same: if the sides $a,b,c$ follow each other in a counterclockwise order, it is positive, if the sides $a,b,c$ follow each other in a clockwise order, it is negative.
		\end{enumerate}
\end{lem}

Using Lemma \ref{lem:hollow}, we can define the \emph{orientation} of the ordered triple of pairwise intersecting convex sets $A,B,C$, denoted by $\o(ABC)$, as the orientation $\o(xyz)$ of any three points $x\in B\cap C$, $y\in A\cap C$ and $z\in A\cap B$ (see Figure \ref{fig:hollow}). 
We also define the orientation of three convex sets with a common intersection as \emph{zero}.\footnote{It might seem counterintuitive that the intersecting case is assigned 0 but this is the natural choice in some cases; see also \cite[Section 4]{PTa}.}
This way we can assign an orientation to any three members of an intersecting family of convex sets in the plane, which determines their order type. 
We write $\o(ABC)=+1$, $\o(ABC)=-1$, $\o(ABC)=0$, respectively, for positive, negative, zero orientations.
From the definitions, it follows that $\o(ABC)=\o(CAB)=\o(BCA)=-\o(ACB)=-\o(BAC)=-\o(CBA)$. Thus, for any family of pairwise intersecting convex sets in the plane $\mathcal{F}$, this is indeed an orientation, and in case no three elements of $\mathcal{F}$ have a common intersection, it is a total orientation, otherwise it is only partial. 
As we will see later (Lemma \ref{lem:int}), this orientation also satisfies the interiority condition, so it is in fact an order of rank 3.
We will denote a 3-order that is the order type of a collection of pairwise intersecting convex sets by \CPO, and by \CTO if no three sets have a common point.
These are compared to different 3-orders of interest in Figure \ref{fig:diagram}.

Since throughout most of the paper, we will use the specific orientation defined above, $\o(ABC)$ will always denote the orientation of pairwise intersecting convex sets in the plane $A$, $B$ and $C$ in this sense (similarly to the orientation of points).

\begin{figure}[!t]
	\centering
	\input{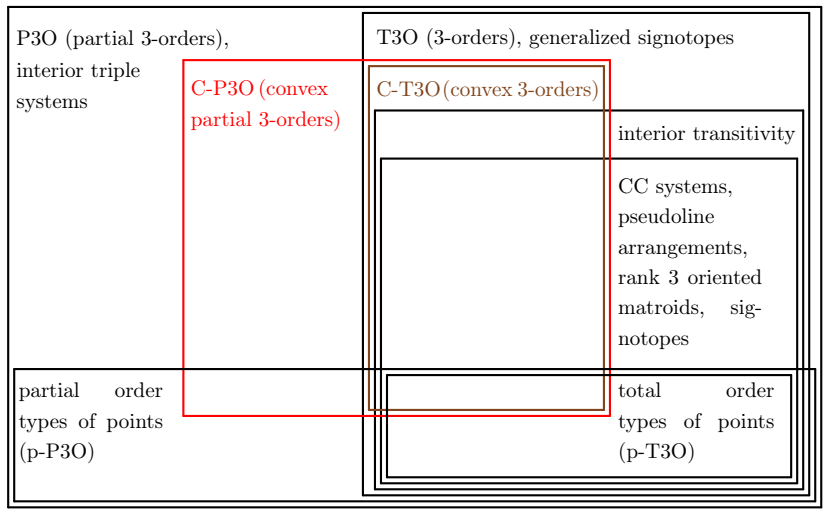}    
	\caption{A diagram illustrating the relationship of some 3-orders. A \PO is any partial orientation of triples satisfying 
	the interiority condition. A \TO is a \PO such that no triple is zero-oriented. A \CTO (resp.\ \CPO) is a \TO (resp.\ \PO) that is realizable with planar convex sets. Theorem \ref{thm:5point} shows that a $\pPO$ might not be a $\CPO$, while that a $\pTO$ might not be a $\CTO$ is proved in our companion paper \cite{GC-3TO}, which also includes several further subclasses of \PO and \TO.\protect\footnotemark~These also imply that \CPO's are a proper subclass of the \PO's and \CTO's are a proper subclass of the \TO's.}\label{fig:diagram}
\end{figure}

\footnotetext{For the Bisztriczky-Fejes T\'oth type definition of order types of convex sets, any point order type is by definition realizable by convex sets, while in the other direction a configuration of convex sets whose order type is not realizable by points was given in \cite{PT3} answering a question of Hubard and Montejano.}

\begin{remark}
	Our definition only allows us to define an orientation for pairwise intersecting triples of convex sets.
	This is unlike the situation in the case of the (quite different) definition of orientation from \cite{BFT1,BFT2,BFT3} by Bisztriczky and Fejes T\'oth (later also investigated in \cite{DHH1,DHH2,HMMS,PT1,PT2,PT3,Suk}) which primarily focused on Erd\H{o}s--Szekeres type theorems.\footnote{For intersecting families, an Erd\H{o}s--Szekeres type theorem with our definition of orientation follows directly from Ramsey's theorem.}
	In these papers the condition on the family of convex sets is that they are pairwise disjoint, or in later papers that they are non-crossing.
	Such a family is in convex position if no set is covered by the convex hull of the rest.
	In this case the orientation of some ordered triple $A,B,C$ is determined by any points $a\in A, b\in B, c\in C$ chosen from the boundary of $conv(A\cup B\cup C)$.
	This definition appeared explicitly in \cite{HMMS} and is implicitly in earlier works---we will refer to it as the Bisztriczky--Fejes T\'oth type orientation.
	Note that if $A,B,C$ are in addition also intersecting but 3-intersection-free, then the Bisztriczky--Fejes T\'oth type definition gives the same orientation as the one used in this paper.
	But such families can contain at most four connected sets, as $K_5$ is non-planar.
\end{remark}

We state the following corollary of
Lemma \ref{lem:hollow}(d).

\begin{cor}\label{cor:subset}
	If the convex sets $A,B,C$ do not have a point in common, and the convex sets $A'\subset A, B'\subset B, C'\subset C$ are pairwise intersecting, then $\o(A'B'C')=\o(ABC)$.
\end{cor}

If a family of convex sets in the plane is intersecting and 3-intersection-free, we call it \emph{holey}.\footnote{Holey families can be also defined in a different way for abstract families, see the recent result in extremal combinatorics \cite{NP}, but our notion is quite different.} 
Thus, the order type of holey family is always a T3O.
For example, any collection of lines in general position is holey, and the orientation of any triple is determined by their slopes (see Figure \ref{fig:line_orient}).
This orientation for lines is not to be confused with the much studied arrangement types of lines which were shown by Goodman and Pollack \cite{GP82} to correspond to the order types of points by duality.
However, Goodman and Pollack also made the following simple observation about the orientations of triples of lines, which is relevant for us.

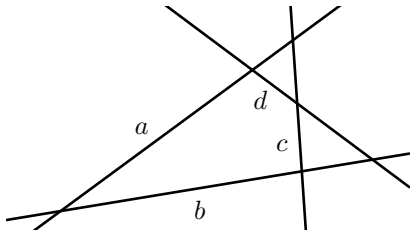
\begin{figure}[!ht]
	\centering
	\begin{tikzpicture}[line cap=round,line join=round,>=triangle 45,x=1cm,y=7mm]
	\clip(-1.0609303135825612,-0.11570559856350218) rectangle (4.29537626950815,4.10958804068733);
	\draw [line width=1.pt,domain=-1.0609303135825612:4.29537626950815] plot(\x,{(--1.6064--2.86*\x)/2.72});
	\draw [line width=1.pt,domain=-1.0609303135825612:4.29537626950815] plot(\x,{(-12.228--4.2*\x)/-0.2});
	\draw [line width=1.pt,domain=-1.0609303135825612:4.29537626950815] plot(\x,{(-14.32--2.9*\x)/-2.72});
	\draw [line width=1.pt,domain=-1.0609303135825612:4.29537626950815] plot(\x,{(--1.0832--0.82*\x)/3.48});
	\draw (0.5182198344202782,2.082300688521527) node[anchor=north west] {$a$};
	\draw (1.3077949084216978,0.6) node[anchor=north west] {$b$};
	\draw (2.4,1.762202685547979) node[anchor=north west] {$c$};
	\draw (2.1,2.7) node[anchor=north west] {$d$};
	\end{tikzpicture}
	\caption{The following triples have positive orientation: $abc$, $abd$, $adc$, $bdc$.}
	\label{fig:line_orient}
\end{figure}

\begin{obs}[Goodman-Pollack \cite{GP84}]
If a holey family consists of lines $l_1,...,l_n$ where the lines are ordered according to their slopes in clockwise circular order, then $\o(l_i,l_j,l_k)=+1$ for all integers $1\le i<j<k\le n$.

\end{obs}

Our main motivation to study holey families is that it can be the first step to improve our understanding of the intersection structure of planar convex sets, which can potentially lead to improved weak $\eps$-nets \cite{ABFK} and $(p,q)$-theorems \cite{AK}.
The question is, what abstract properties of the underlying geometric 3-hypergraphs are useful to derive interesting results.

The rest of this paper is organized as follows.
In Section \ref{sec:interiority} we show that \CPO's satisfy the interiority condition (meaning that they are indeed \PO's as the name suggests, and similarly \CTO's are \TO's), and compare them with other well-studied orientations.
In Section \ref{app:small} we examine which \pTO's (order types) are realizable as a \CTO, and we find that up to five elements, the single condition that the configuration satisfies the interiority condition, is sufficient.
On the other hand, in Section \ref{app:3int} we show that there is a five-element \pPO (five-point order type with collinearities) that cannot be realized as a \CPO.
We prove the strengthening that there is a \pTO that is not a \CTO in a companion paper \cite{GC-3TO}, which primarily studies orientations of \emph{good covers}, a generalization of the orientation studied here.
In Section \ref{app:43} we study what happens if a \CPO also has the (4,3) property, that is, among any four convex sets there are three that have a common point.
We derive some new abstract properties, however, we also show that on their own they are not yet sufficient to prove a $(p,q)$-theorem.
Finally, in Section \ref{sec:discussion} we pose some open problems.

We end this Introduction with a comparison to other works.

\subsubsection*{Other works studying 3-orders}

Knuth \cite[Chapter 3]{Knuth} studied orientations 
that satisfy the interiority condition under the name \emph{interior triple system}, according to Knuth ``for want of a better name.''
We want a better name, so we will refer to such an orientation as a \TO (total 3-order), while if zero-orientations are also allowed, then we call such an orientation a \PO (partial 3-order).
We believe that these names are better as they reflect the similarity to posets, which would be called a \PtwoO (partial 2-order) in our language.
The main result in \cite[Chapter 3]{Knuth} is that there are $2^{\Omega(n^3)}$ different \TO's over $n$ elements.

As we have learned after the first preprint of our paper already appeared, Bergold et al. \cite{BFSchSchS} have also studied \TO's.
Unaware of Knuth's work, they named them \emph{generalized signotopes}, and studied primarily a special subclass of them that can be defined from topological drawings of $K_n$.
Their equivalent definition resembled the definition of signotopes:
They said that an orientation is a 3-order if on any four of its elements the sequence $\o(ABC),\o(ABD),\o(ACD),\o(BCD)$ is not $+1,-1,+1,-1$ or $-1,+1,-1,+1$, i.e., it can change sign at most twice.
It can be easily seen that this is exactly the same as our definition of \TO.
They studied the appropriately defined versions of well-known results and proved, for example, that Kirchberger's theorem holds for \TO but Helly's theorem does not.
They also rediscovered Knuth's construction that there are $2^{\Omega(n^3)}$ different \TO's, and with a recursion obtained a slightly better constant and also an upper bound, proving that the number of \TO is between $2^{0.25\binom n3+o(n^3)}$ and $2^{0.84\binom n3+o(n^3)}$.

To the best of our knowledge, \TO's have not been studied anywhere except the above mentioned two places.
However, if we add another property, called \emph{transitivity} (the definition of which we omit here), then we get a much better studied notion, known as CC systems \cite{Knuth} (see also pseudoline arrangements \cite[Chapter 5]{Handbook} and acyclic rank 3 oriented matroids \cite[Chapter 6]{Handbook}).
The transitivity property, however, is not satisfied by holey convex families.
In fact, not even the following weaker condition, that we define below.

Knuth \cite[Chapter 2, (2.4)]{Knuth} defines the \emph{interior transitivity} condition as follows: If $D\in conv(ABC)$ and $E\in conv(ABD)$, then $E\in conv(ABC)$ (recall that $W\in conv(XYZ)$ means $\o(XYW)=\o(YZW)=\o(ZXW)\neq0$).
The interior transitivity condition is satisfied by the earlier mentioned CC systems, but it is strictly weaker than them.
Indeed, Knuth proved that the number of 3-orders on $n$ sets is $2^{\Omega(n^2\log n)}$, while the number of CC systems is $2^{\Theta(n^2)}$, and Goodman and Pollack proved that the number of CC systems that are representable by planar point sets, known as stretchable arrangements/order types, is $2^{\Theta(n\log n)}$. \cite{GP93}
There are holey families that do not satisfy the interior transitivity condition, see Figure \ref{fig:interiortrans}.

\begin{figure}[!h]
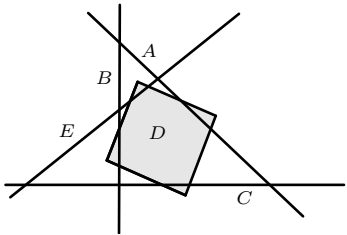

    \centering
    \include{Figures/fura}
    \caption{A family of convex sets not satisfying the interior transitivity: $D\in conv(ABC)$ and $E\in conv(ABD)$ but $E\notin conv(ABC)$.}
    \label{fig:interiortrans}
\end{figure}

However, the following weaker statement is true.

\begin{claim}\label{claim:subknuth}
	If $A, B, C, D$ and $E$ are convex sets forming a holey family such that $D\in conv(ABC)$ and $E\in conv(ABD)$, then $D\cap E\subset \hollow(ABC)$. 
\end{claim}

For the proof, we need the following simple observation, which follows from checking how a regular or irregular containment can look like; see Figure \ref{fig:containment}.

\begin{obs}\label{obs:deabc}
    Suppose $A, B, C$ and $O$ are elements of a holey family.\\
	Then $O\in conv(ABC)$ if and only if $\hollow(ABO)$, $\hollow(BCO)$, $\hollow(CAO)\subset \hollow(ABC)\cup A\cup B\cup C$.
\end{obs}

\begin{proof}[Proof of Claim \ref{claim:subknuth}]
	Since $D\cap E$ intersects $\partial \hollow(ABD)$ which is contained in $\hollow(ABC)\cup A\cup B\cup C$ by Observation \ref{obs:deabc}, and $D\cap E$ cannot intersect $A\cup B\cup C$ as there are no triple intersections, we get that $D\cap E\subset \hollow(ABC)$, as required.
\end{proof}

\section{Interiority}\label{sec:interiority}

We already defined the interiority condition for orientations defined on families of $k$-tuples. Now we will prove that \CPO's do satisfy the interiority condition and thus they are \PO's. This obviously also proves that \CTO's are \TO's.

Recall that we say that a (partial) orientation satisfies the \emph{interiority condition} if $\o(ABD)={\o(BCD)}=\o(CAD)=1$ imply $\o(ABC)=1$ for any $A, B, C, D$. 
But some others, like separability, are preserved; see \cite{BFSchSchS} for a proof of Kirchberger's Theorem.

\begin{lem}[Interiority Lemma]\label{lem:int}
For any four pairwise intersecting convex sets $A, B, C, D$ if $\o(ABD)=\o(BCD)=\o(CAD)=1$, then $\o(ABC)=1$.
\end{lem}

Within this section, we will also use the following notation:

\begin{proof}
Suppose $A,B,C,D$ is an intersecting family of convex sets and $\o(ABD)=\o(BCD)=\o(CAD)=1$. Then we need to show that $\o(ABC)=1$.

	For a contradiction, suppose first $\o(ABC)=0$. Fix some $w\in A\cap B\cap C$, and take any $a\in A\cap D$, $b\in B\cap D$ and $c\in C\cap D$ and check the orientations of the triples of $w,a,b,c$ using Lemma \ref{lem:hollow}(d). We get $\o(abw)=\o(bcw)=\o(caw)=1$.	It follows that $w\in conv(a,b,c)\subset D$, contradicting that $\o(ABD)=1$.
	
	Now suppose $\o(ABC)=-1$. Take any $a\in A\cap D$, $b\in B\cap D$, $c\in C\cap D$, $z\in A\cap B$, $x\in B\cap C$ and $y\in A\cap C$.
	We can assume that these six points are in general position, otherwise we could slightly perturb them, along with the convex sets containing them, if necessary, without introducing a triple intersection.
	The conditions and Lemma \ref{lem:hollow}(d) imply that $\o(abz)=\o(bcx)=\o(cay)=-1$ and $\o(xyz)=-1$.  
	Also, as there is no triple intersection, we know that $x,y,z\notin conv(abc)$, $b,c,x\notin conv(ayz)$, 
	$a,c,y\notin conv(bxz)$, 
	$a,b,z\notin conv(cxy)$.
	We will deal with two cases, depending on the orientation of $abc$.
	The lines $ab,bc,ca$ divide the plane into seven regions:  a bounded triangle $conv(abc)$, three unbounded cones, which we denote by $V_a, V_b, V_c$, respectively, indexed by their apexes, and three unbounded regions sharing a side each with the triangle $conv(abc)$, which we denote by $U_{ab}, U_{bc}, U_{ac}$, respectively, indexed by the adjacent side of the triangle.
	
	\begin{figure}[!ht]
		\centering
		\definecolor{ffqqqq}{rgb}{1.,0.,0.}
		\scalebox{0.8}{\begin{tikzpicture}[line cap=round,line join=round,>=triangle 45,x=1.0cm,y=1.0cm]
			\clip(-2.063588759444999,-1.6963559121803733) rectangle (7.388920672533251,7.059194299006854);
			\fill[line width=2.pt,fill=black,fill opacity=0.10000000149011612] (-1.,1.) -- (1.34,3.06) -- (3.,5.) -- (3.,6.) -- (-1.,6.) -- cycle;
			\fill[line width=2.pt,fill=black,fill opacity=0.10000000149011612] (2.,2.) -- (-1.,2.) -- (-1.,-1.) -- (6.,-1.) -- (6.,1.) -- cycle;
			\fill[line width=2.pt,fill=black,fill opacity=0.10000000149011612] (6.,0.) -- (4.,2.) -- (3.04,3.1) -- (2.,6.) -- (6.,6.) -- cycle;
			\fill[line width=2.pt,fill=black,fill opacity=0.10000000149011612] (2.2,3.9) -- (0.5,4.) -- (2.,1.) -- (4.,3.8) -- cycle;
			\draw [line width=1.pt,color=ffqqqq,domain=-2.063588759444999:7.388920672533251] plot(\x,{(--8.75-0.*\x)/2.5});
			\draw [line width=1.pt,color=ffqqqq,domain=-2.063588759444999:7.388920672533251] plot(\x,{(-1.75--2.*\x)/1.5});
			\draw [line width=1.pt,color=ffqqqq,domain=-2.063588759444999:7.388920672533251] plot(\x,{(-5.5--2.*\x)/-1.});
			\draw [line width=1.1pt] (-1.,1.)-- (1.34,3.06);
			\draw [line width=1.1pt] (1.34,3.06)-- (3.,5.);
			\draw [line width=1.1pt] (3.,5.)-- (3.,6.);
			\draw [line width=1.1pt] (3.,6.)-- (-1.,6.);
			\draw [line width=1.1pt] (-1.,6.)-- (-1.,1.);
			\draw [line width=1.1pt] (2.,2.)-- (-1.,2.);
			\draw [line width=1.1pt] (-1.,2.)-- (-1.,-1.);
			\draw [line width=1.1pt] (-1.,-1.)-- (6.,-1.);
			\draw [line width=1.1pt] (6.,-1.)-- (6.,1.);
			\draw [line width=1.1pt] (6.,1.)-- (2.,2.);
			\draw [line width=1.1pt] (6.,0.)-- (4.,2.);
			\draw [line width=1.1pt] (4.,2.)-- (3.04,3.1);
			\draw [line width=1.1pt] (3.04,3.1)-- (2.,6.);
			\draw [line width=1.1pt] (2.,6.)-- (6.,6.);
			\draw [line width=1.1pt] (6.,6.)-- (6.,0.);
			\draw [line width=1.1pt] (2.2,3.9)-- (0.5,4.);
			\draw [line width=1.1pt] (0.5,4.)-- (2.,1.);
			\draw [line width=1.1pt] (2.,1.)-- (4.,3.8);
			\draw [line width=1.1pt] (4.,3.8)-- (2.2,3.9);
			\draw [color=ffqqqq] (6.278141914397557,6) node[anchor=north west] {$V_a$};
			\draw [color=ffqqqq] (-1.8675689785975236,6) node[anchor=north west] {$V_b$};
			\draw [color=ffqqqq] (1.464767295809555,-0.9558367400899113) node[anchor=north west] {$V_c$};
			\draw [color=ffqqqq] (1.9657067357531028,7.0) node[anchor=north west] {$U_{ab}$};
			\draw [color=ffqqqq] (6.343481841346716,1.0) node[anchor=north west] {$U_{ac}$};
			\draw [color=ffqqqq] (-1.8893489542472433,1.0) node[anchor=north west] {$U_{bc}$};
			\draw (0.95,4.0) node[anchor=north west] {$b$};
			\draw (3.00,3.85) node[anchor=north west] {$a$};
			\draw (2.05,1.8) node[anchor=north west] {$c$};
			\draw (2.4230862243972116,5.926635565221441) node[anchor=north west] {$z$};
			\draw (-0.8874700743601476,1.9844599726222172) node[anchor=north west] {$x$};
			\draw (5.5,1.0) node[anchor=north west] {$y$};
			
			\draw [fill=black] (3.5,3.5) circle (1.5pt);
			\draw [fill=black] (1.0,3.5) circle (1.5pt);
			\draw [fill=black] (2.0,1.53) circle (1.5pt);
			
			\draw [fill=black] (-0.455,1.8) circle (1.5pt);
			\draw [fill=black] (5.5,0.9390211414356826) circle (1.pt);
			\draw [fill=black] (2.65,5.45) circle (1.5pt);
			\end{tikzpicture}}
		\caption{Case 1 of the proof of Lemma \ref{lem:int}. Beware that in the figure $\o(xyz)=1$ while in the proof $\o(xyz)=-1$ but we could find no better way to depict contradicting assumptions.}
		\label{fig:sevenpart}
	\end{figure}
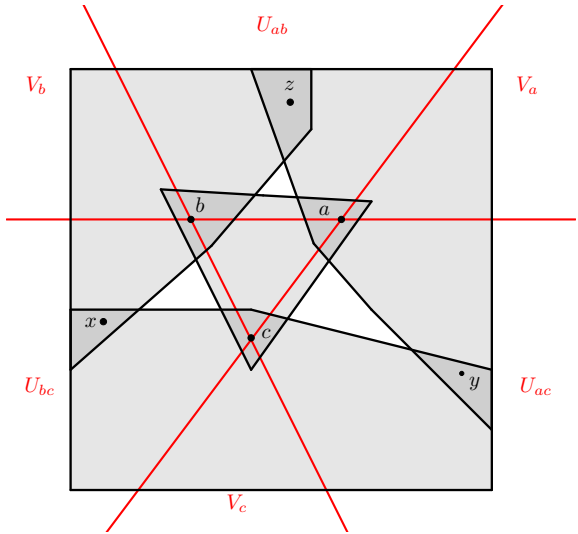
	
	Case 1: $\o(abc)=1$ (see Figure \ref{fig:sevenpart}).\\
	The orientation conditions and $x,y,z\notin conv(abc)$ imply that
	$x\in V_b\cup U_{bc}\cup V_c$,
	$y\in V_c\cup U_{ac}\cup V_a$,
	$z\in V_a\cup U_{ab}\cup V_b$.\\
	Since $\o(xyz)=-1$, two of $x,y,z$ must fall in the same cone $V_i$.
	Without loss of generality, assume that $x,y\in V_c$.
	As $c\notin conv(ayz)$, and $a$ is to the right of the directed line $yc$, $z$ must either lie to the right of line $yc$ or to the left of the line $ac$. Since $z$ lies to right of the line $ab$, if it lies to the left of $ac$ then it is in $V_a$. Hence $z$ must lie to the right of $yc$. Similarly $z$ must lie to the left of $xc$. But this implies $z\in U_{ab}$ and $\o(xyz)=1$, a contradiction.
	
	Case 2: $\o(abc)=-1$.\\
	The orientation conditions and $x,y,z\notin conv(abc)$ imply that
	$x\in U_{ab}\cup V_a\cup U_{ac}$,
	$y\in U_{bc}\cup V_b\cup U_{ab}$,
	$z\in U_{ac}\cup V_c\cup U_{bc}$.\\
	If any of $x,y,z$ fall in a cone $V_i$, e.g., $x$ falls in $V_a$ then $y,z\in U_{a,c}$ and we can finish with a similar argument as in the previous case.\\
	Otherwise, say that $a$ has an \emph{opposite} point, if $y\in U_{bc}$ or $z\in U_{bc}$ and, similarly, $b$ has an opposite point, if $x\in U_{ac}$ or $y\in U_{ac}$ and $c$ has an opposite point, if $x\in U_{ab}$ or $y\in U_{ab}$.
	If $a$ does not have an opposite point, then $y\in U_{ab}$ and $z\in U_{ac}$, which implies that both $b$ and $c$ have an opposite point.
	Therefore, at least two of $a,b,c$ have an opposite point, say, $b$ and $c$.
	But then the segments connecting $b$ and $c$ to their opposite points intersect inside $conv(abc)$, which gives a triple intersection, contradicting our assumptions.
\end{proof}

\subsubsection*{Regular and irregular containment}
In Figure \ref{fig:containment} we can see two different ways $D\in conv(ABC)$ can happen. We will see that the one on the right complicates many scenarios, so we will often handle the two cases separately.  We say that the containment $D\in conv(ABC)$ is \emph{regular} if each of $D\cap \partial\hollow(ABC)\cap A$, $D\cap \partial\hollow(ABC)\cap B$ and $D\cap \partial\hollow(ABC)\cap C$ is a connected set, and we say that the containment $D\in conv(ABC)$ is \emph{irregular} if one of them has more than one connected component. Also, whichever of $A$, $B$ and $C$ has a disconnected intersection with $D\cap\partial\hollow(ABC)$, we call the containment irregular with respect to that set.
If $D\in conv(ABC)$ is regular, then each of $D\cap \hollow (ABC) \cap  \partial A$, $D\cap \hollow (ABC) \cap  \partial B$ and $D\cap \hollow (ABC) \cap  \partial C$ is a connected curve.

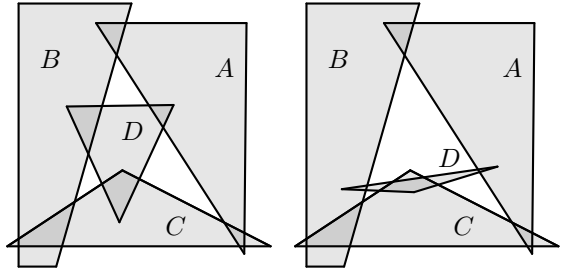
\begin{figure}[t]
	\centering
	\begin{tikzpicture}[line cap=round,line join=round,>=triangle 45,x=0.5cm,y=0.5cm]
	\clip(-2.34,-1.72) rectangle (5.08,5.9);
	\fill[line width=0.8pt,fill=black,fill opacity=0.10000000149011612] (-2.,-1.) -- (5.,-1.) -- (1.06,1.02) -- cycle;
	\fill[line width=0.8pt,fill=black,fill opacity=0.10000000149011612] (-1.7,-1.54) -- (-1.7,5.46) -- (1.3,5.46) -- (-0.7,-1.54) -- cycle;
	\fill[line width=0.8pt,fill=black,fill opacity=0.10000000149011612] (4.3,-1.2) -- (0.36,4.94) -- (4.36,4.94) -- cycle;
	\fill[line width=0.8pt,fill=black,fill opacity=0.10000000149011612] (-0.42,2.72) -- (0.98,-0.36) -- (2.42,2.76) -- cycle;
	\draw [line width=0.8pt] (1.06,1.02)-- (-2.,-1.);
	\draw [line width=0.8pt] (1.06,1.02)-- (5.,-1.);
	\draw [line width=0.8pt] (-2.,-1.)-- (5.,-1.);
	\draw [line width=0.8pt] (5.,-1.)-- (1.06,1.02);
	\draw [line width=0.8pt] (1.06,1.02)-- (-2.,-1.);
	\draw [line width=0.8pt] (-1.7,-1.54)-- (-1.7,5.46);
	\draw [line width=0.8pt] (-1.7,5.46)-- (1.3,5.46);
	\draw [line width=0.8pt] (1.3,5.46)-- (-0.7,-1.54);
	\draw [line width=0.8pt] (-0.7,-1.54)-- (-1.7,-1.54);
	\draw [line width=0.8pt] (4.3,-1.2)-- (0.36,4.94);
	\draw [line width=0.8pt] (0.36,4.94)-- (4.36,4.94);
	\draw [line width=0.8pt] (4.36,4.94)-- (4.3,-1.2);
	\draw [line width=0.8pt] (-0.42,2.72)-- (0.98,-0.36);
	\draw [line width=0.8pt] (0.98,-0.36)-- (2.42,2.76);
	\draw [line width=0.8pt] (2.42,2.76)-- (-0.42,2.72);
	\draw (-1.38,4.46) node[anchor=north west] {$B$};
	\draw (3.26,4.24) node[anchor=north west] {$A$};
	\draw (1.96,0.02) node[anchor=north west] {$C$};
	\draw (0.78,2.56) node[anchor=north west] {$D$};
	\end{tikzpicture}
	\begin{tikzpicture}[line cap=round,line join=round,>=triangle 45,x=0.5cm,y=0.5cm]
	\clip(-2.34,-1.72) rectangle (5.08,5.9);
	\fill[line width=0.8pt,fill=black,fill opacity=0.10000000149011612] (-2.,-1.) -- (5.,-1.) -- (1.06,1.02) -- cycle;
	\fill[line width=0.8pt,fill=black,fill opacity=0.10000000149011612] (-1.7,-1.54) -- (-1.7,5.46) -- (1.3,5.46) -- (-0.7,-1.54) -- cycle;
	\fill[line width=0.8pt,fill=black,fill opacity=0.10000000149011612] (4.3,-1.2) -- (0.36,4.94) -- (4.36,4.94) -- cycle;
	\fill[line width=0.8pt,fill=black,fill opacity=0.10000000149011612] (-0.76,0.52) -- (1.16,0.44) -- (3.38,1.12) -- cycle;
	\draw [line width=0.8pt] (1.06,1.02)-- (-2.,-1.);
	\draw [line width=0.8pt] (1.06,1.02)-- (5.,-1.);
	\draw [line width=0.8pt] (-2.,-1.)-- (5.,-1.);
	\draw [line width=0.8pt] (5.,-1.)-- (1.06,1.02);
	\draw [line width=0.8pt] (1.06,1.02)-- (-2.,-1.);
	\draw [line width=0.8pt] (-1.7,-1.54)-- (-1.7,5.46);
	\draw [line width=0.8pt] (-1.7,5.46)-- (1.3,5.46);
	\draw [line width=0.8pt] (1.3,5.46)-- (-0.7,-1.54);
	\draw [line width=0.8pt] (-0.7,-1.54)-- (-1.7,-1.54);
	\draw [line width=0.8pt] (4.3,-1.2)-- (0.36,4.94);
	\draw [line width=0.8pt] (0.36,4.94)-- (4.36,4.94);
	\draw [line width=0.8pt] (4.36,4.94)-- (4.3,-1.2);
	\draw [line width=0.8pt] (-0.76,0.52)-- (1.16,0.44);
	\draw [line width=0.8pt] (1.16,0.44)-- (3.38,1.12);
	\draw [line width=0.8pt] (3.38,1.12)-- (-0.76,0.52);
	\draw (-1.38,4.46) node[anchor=north west] {$B$};
	\draw (3.26,4.24) node[anchor=north west] {$A$};
	\draw (1.96,0.02) node[anchor=north west] {$C$};
	\draw (1.55,1.85) node[anchor=north west] {$D$};
	\end{tikzpicture}\caption{Regular and irregular containment $D\in conv(ABC)$.}
	\label{fig:containment}
\end{figure}

\subsubsection*{Doubly irregular containments are impossible}

\begin{claim}
For convex sets $A$, $B$, $C$ and $D$, it is impossible that $D\in conv(ABC)$ and the containment is irregular with respect to both $A$ and $B$.
\end{claim}

\begin{proof}
Suppose $D$ is irregular with respect to $A$ and $B$. By definition, we know that $D\setminus int(A)$ has at least two connected components, whose border with $A$ is part of $\partial{\hollow(ABC)}$; let us call these components $D_{A1}$ and $D_{A2}$. Similarly, define $D_{B1}$ and $D_{B2}$. Since $D\cap int(A)$ and $D\cap int(B)$ are disjoint (and both are connected), $D\cap int(B)$ is separated by $D\cap int(A)$ within $D$ from at least one of $D_{A1}$ and $D_{A2}$. Similarly, $D\cap int(A)$ is separated by $D\cap int(B)$ within $D$ from one of $D_{B1}$ and $D_{B2}$. We can assume without loss of generality that these components are $D_{A1}$ and $D_{B1}$, respectively (see Figure \ref{fig:dirreg}). 
Choose two points, $p$ from $D_{A1}\cap  \partial{\hollow(ABC)}$ and $q$ from $D_{B1}\cap \partial{\hollow(ABC)}$. The segment between $p$ and $q$ crosses both $int(A)$ and $int(B)$, and has both of its endpoints within $\partial{\hollow(ABC)}$. Thus, the $pq$ line can leave ${\hollow(ABC)}$ (which is a bounded set) only via $C$ in both directions. But since the $pq$ segment is not fully contained in $C$, and $C$ is convex, we get a contradiction.
\end{proof}
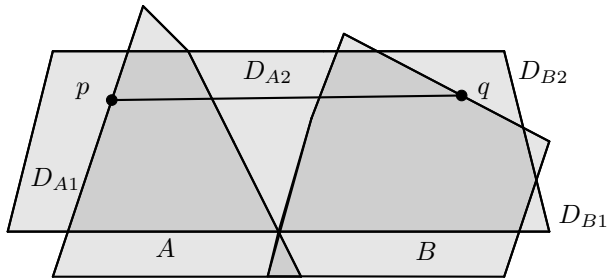
\begin{figure}[h]
    \centering
    \definecolor{zzttqq}{rgb}{0.6,0.2,0.}
\begin{tikzpicture}[line cap=round,line join=round,>=triangle 45,x=0.6cm,y=0.6cm]
\clip(-0.3239038565299793,-3.6507200023489927) rectangle (13.409495775404775,3.6161371455511833);
\fill[line width=2.pt,fill=black,fill opacity=0.10000000149011612] (1.,-3.) -- (3.,3.) -- (4.,2.) -- (5.,0.) -- (6.,-2.) -- (6.5,-3.) -- cycle;
\fill[line width=1.pt,fill=black,fill opacity=0.10000000149011612] (5.761176214858738,-2.990541253882303) -- (6.730055510618764,0.510670070861791) -- (7.449354863662415,2.3859862412970236) -- (12.,0.) -- (11.,-3.) -- cycle;
\fill[line width=1.pt,fill=black,fill opacity=0.10000000149011612] (0.,-2.) -- (1.,2.) -- (11.,2.) -- (12.,-2.) -- cycle;
\draw [line width=0.8pt] (1.,-3.)-- (3.,3.);
\draw [line width=0.8pt] (3.,3.)-- (4.,2.);
\draw [line width=0.8pt] (4.,2.)-- (5.,0.);
\draw [line width=0.8pt] (6.730055510618764,0.510670070861791)-- (7.449354863662415,2.3859862412970236);
\draw [line width=0.8pt] (7.449354863662415,2.3859862412970236)-- (12.,0.);
\draw [line width=0.8pt] (6.730055510618764,0.510670070861791)-- (6.,-2.);
\draw [line width=0.8pt] (5.,0.)-- (6.,-2.);
\draw [line width=0.8pt] (12.,0.)-- (11.,-3.);
\draw [line width=0.8pt] (12.,-2.)-- (11.,2.);
\draw [line width=0.8pt] (11.,2.)-- (7.,2.);
\draw [line width=0.8pt] (7.,2.)-- (1.,2.);
\draw [line width=0.8pt] (1.,2.)-- (0.,-2.);
\draw [line width=0.8pt] (0.,-2.)-- (12.,-2.);
\draw [line width=0.8pt] (6.,-2.)-- (5.761176214858738,-2.990541253882303);
\draw [line width=0.8pt] (6.,-2.)-- (6.5,-3.);
\draw (3.0694303182604257,-1.9540529149537973) node[anchor=north west] {$A$};
\draw (8.83169589809319,-2.0180780880630502) node[anchor=north west] {$B$};
\draw (0.3,-0.38543617377710737) node[anchor=north west] {$D_{A1}$};
\draw (12,-1.2) node[anchor=north west] {$D_{B1}$};
\draw (5,2) node[anchor=north west] {$D_{A2}$};
\draw (11.1,2) node[anchor=north west] {$D_{B2}$};
\draw (1.3087380577559704,1.5) node[anchor=north west] {$p$};
\draw (10.2,1.5) node[anchor=north west] {$q$};
\draw [line width=0.8pt] (1.,-3.)-- (3.,3.);
\draw [line width=0.8pt] (3.,3.)-- (4.,2.);
\draw [line width=0.8pt] (4.,2.)-- (5.,0.);
\draw [line width=0.8pt] (5.,0.)-- (6.,-2.);
\draw [line width=0.8pt] (6.,-2.)-- (6.5,-3.);
\draw [line width=0.8pt] (6.5,-3.)-- (1.,-3.);
\draw [line width=0.8pt] (5.761176214858738,-2.990541253882303)-- (6.730055510618764,0.510670070861791);
\draw [line width=0.8pt] (6.730055510618764,0.510670070861791)-- (7.449354863662415,2.3859862412970236);
\draw [line width=0.8pt] (7.449354863662415,2.3859862412970236)-- (12.,0.);
\draw [line width=0.8pt] (12.,0.)-- (11.,-3.);
\draw [line width=0.8pt] (11.,-3.)-- (5.761176214858738,-2.990541253882303);
\draw [line width=0.8pt] (2.305626368583439,0.9168791057503182)-- (10.057779342713298,1.018341713539659);
\draw [line width=0.8pt] (0.,-2.)-- (1.,2.);
\draw [line width=0.8pt] (1.,2.)-- (11.,2.);
\draw [line width=0.8pt] (11.,2.)-- (12.,-2.);
\draw [line width=0.8pt] (12.,-2.)-- (0.,-2.);
\begin{scriptsize}
\draw [fill=black] (2.305626368583439,0.9168791057503182) circle (2.0pt);
\draw [fill=black] (10.057779342713298,1.018341713539659) circle (2.0pt);
\end{scriptsize}

\end{tikzpicture}\caption{Doubly irregular containments are impossible.}
    \label{fig:dirreg}
\end{figure}

\subsubsection*{Inextendible holey family}\label{app:inext}

We end this section by marking an important difference between point configurations and holey families. While any finite set of points can be extended by adding another arbitrary point, and any \TO is also extendable \cite{BFSchSchS}, this is not the case for holey families.
Note that this does \emph{not} imply that the \CTO of the holey family is not extendable.
We do not know whether all \CTO's are extendable or not, but it is easy to see that all \TO's are extendable.
Indeed, just take any element $A$ and add a new element $A'$ such that $\o(A'BC)=\o(ABC)$ and $\o(AA'B)=\o(AA'C)$ for every $B$ and $C$.
If only $A$ or $A'$ occurs among four sets, they will trivially satisfy the interiority condition, while if the four sets are $A,A',B,C$, then $\o(ABA')=\o(BCA')=\o(CAA')$ is impossible, and if $\o(AA'C)=\o(A'BC)=\o(BAC)=1$, we automatically have $\o(AA'B)=1$.

\begin{thm}
	There exists a holey family of convex sets in the plane that cannot be extended to a larger holey family by adding one more convex set to it.
\end{thm}

\begin{figure}[!ht]
	\centering
	\definecolor{uuuuuu}{rgb}{0.26666666666666666,0.26666666666666666,0.26666666666666666}
	\begin{tikzpicture}[line cap=round,line join=round,>=triangle 45,x=3.0cm,y=3.0cm]
	\clip(6.898382281717928,-3.737341577021789) rectangle (8.1,-2.652694238714492);
	\draw [line width=1.pt,fill=black,fill opacity=0.10000000149011612] (9.,-4.) circle (4.65cm);
	\draw [line width=1.pt,fill=black,fill opacity=0.10000000149011612] (7.5,-1.4019237886466835) circle (4.65cm);
	\draw [line width=1.pt,fill=black,fill opacity=0.10000000149011612] (6.,-4.) circle (4.65cm);
	\draw [line width=1.pt,fill=black,fill opacity=0.10000000149011612] (7.5,-3.133974596215561) circle (1.2262692927881012cm);
	\draw (6.951885065221877,-3.31) node[anchor=north west] {$A$};
	\draw (7.88,-3.31) node[anchor=north west] {$B$};
	\draw (7.107529526324273,-2.65) node[anchor=north west] {$C$};
	\draw (7.42,-3.041805391470473) node[anchor=north west] {$D$};
	\begin{scriptsize}
	\draw [fill=uuuuuu] (7.5,-1.4019237886466835) circle (2.5pt);
	\end{scriptsize}
	\end{tikzpicture}
	\caption{A holey family consisting of four disks that cannot be extended.}\label{fig:holey}
\end{figure}
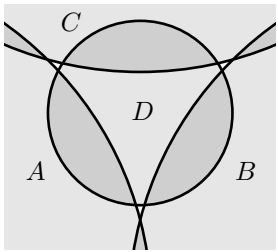

\begin{proof}
	A simple but important observation we use is that in a holey family of at least three convex sets the boundary $\partial X$ of any one of them, $X$, intersects the boundaries of each of the other convex sets in some point not contained in any other of the sets.
	
	Our family consists of four disks, three big, $A,B,C$, and one small, $D$, as depicted in Figure \ref{fig:holey}. If there exists a convex set $X$ such that together with these four disks they form a holey family, then $\partial X$ must intersect $\partial D$ in a point $p$ close to one of the three vertices of $\hollow(ABC)$, say, close to the one which is the intersection of $\partial A$ and $\partial B$. $X$ must also contain some point $q$ of $C$. A simple case analysis shows that the segment $pq$ (contained in $X$ by convexity) contains a point contained in two of the four original disks, creating a triple-intersection with $X$, contradicting that the family is holey.
\end{proof}

\section{Small cases}\label{app:small}

Here we examine which \TO's on few elements are realizable with a holey family of convex sets, similarly as was done in \cite{GP80} for allowable sequences and order types. In case of 4 elements, it follows from Lemma \ref{lem:int} that all system definitions coincide:

\begin{claim}
	On four elements, there are two \pTO's, two \TO's and two \CTO's up to isomorphism.
\end{claim}

In case of 5 elements, a \pTO is determined by the size of the convex hull of the realizing point set, which gives three options, but by enumeration, there are six combinatorially different \TO's.
We could realize all of them with convex sets (see Figures \ref{fig:otponthathalmaz} and \ref{fig:othalmaz2}) which implies:

\begin{claim}
	Any one of the (up to isomorphism) six \TO's on five elements is a \CTO, i.e., it is representable by a holey family of convex sets.
\end{claim}

\begin{figure}[!h]
	\begin{center}
		\begin{minipage}{.25\textwidth}
			\centering
			\input{Figures/5pont1}
		\end{minipage}
		\begin{minipage}{.25\textwidth}
			\centering
			\input{Figures/5pont2}
		\end{minipage}
		\begin{minipage}{.25\textwidth}
			\centering
			\input{Figures/5pont3}
		\end{minipage}
	\end{center}
	
	\begin{center}
		\begin{minipage}{.25\textwidth}
			\centering
			\input{Figures/5halmaz1}
		\end{minipage}
		\begin{minipage}{.25\textwidth}
			\centering
			\input{Figures/5halmaz2}
		\end{minipage}
		\begin{minipage}{.25\textwidth}
			\centering
			\input{Figures/5halmaz3}
		\end{minipage}
		\caption{Three \TO's on five elements can be realized as a \pTO and as a \CTO.}
		\label{fig:otponthathalmaz}
	\end{center}
	
	\begin{center}
		\begin{minipage}{.32\textwidth}
			\centering
			\scalebox{0.8}{\input{Figures/5halmaz4}}
		\end{minipage}
		\begin{minipage}{.32\textwidth}
			\centering
			\scalebox{0.8}{\input{Figures/5halmaz5}}
		\end{minipage}
		\begin{minipage}{.32\textwidth}
			\centering
			\scalebox{0.8}{\input{Figures/5halmaz6}}
		\end{minipage}
		\caption{Three \TO's on five elements can be realized as a \CTO but not as a \pTO.}
		\label{fig:othalmaz2}
	\end{center}
\end{figure}

\begin{figure}[!h]
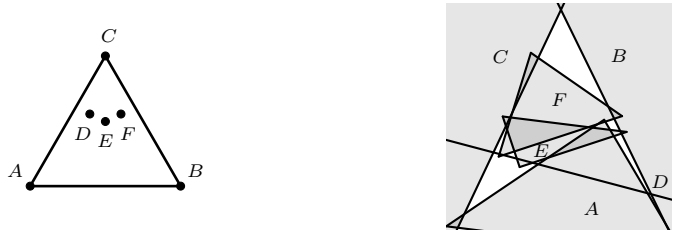

	\begin{center}
		\begin{minipage}{.40\textwidth}
			\centering
			\input{Figures/hatpont}
		\end{minipage}
		\begin{minipage}{.33\textwidth}
			\centering
			\input{Figures/h11}
		\end{minipage}
		\caption{Six-point set and its representation ($D$ is a segment, the other sets triangles).
		Any representation must have an irregular containment; here $D,E\in conv(ABC)$ are both irregulars.}
		\label{fig:hatponthathalmaz}
	\end{center}
\end{figure}

In case of 6 elements, it can be checked by enumeration that in total there are 253 \TO's on 6 elements, which is much more than 16, the number of \pTO's realizable with 6 points, both numbers counted up to isomorphism.

\def\kicsiny{0.95}
\begin{figure}[p]
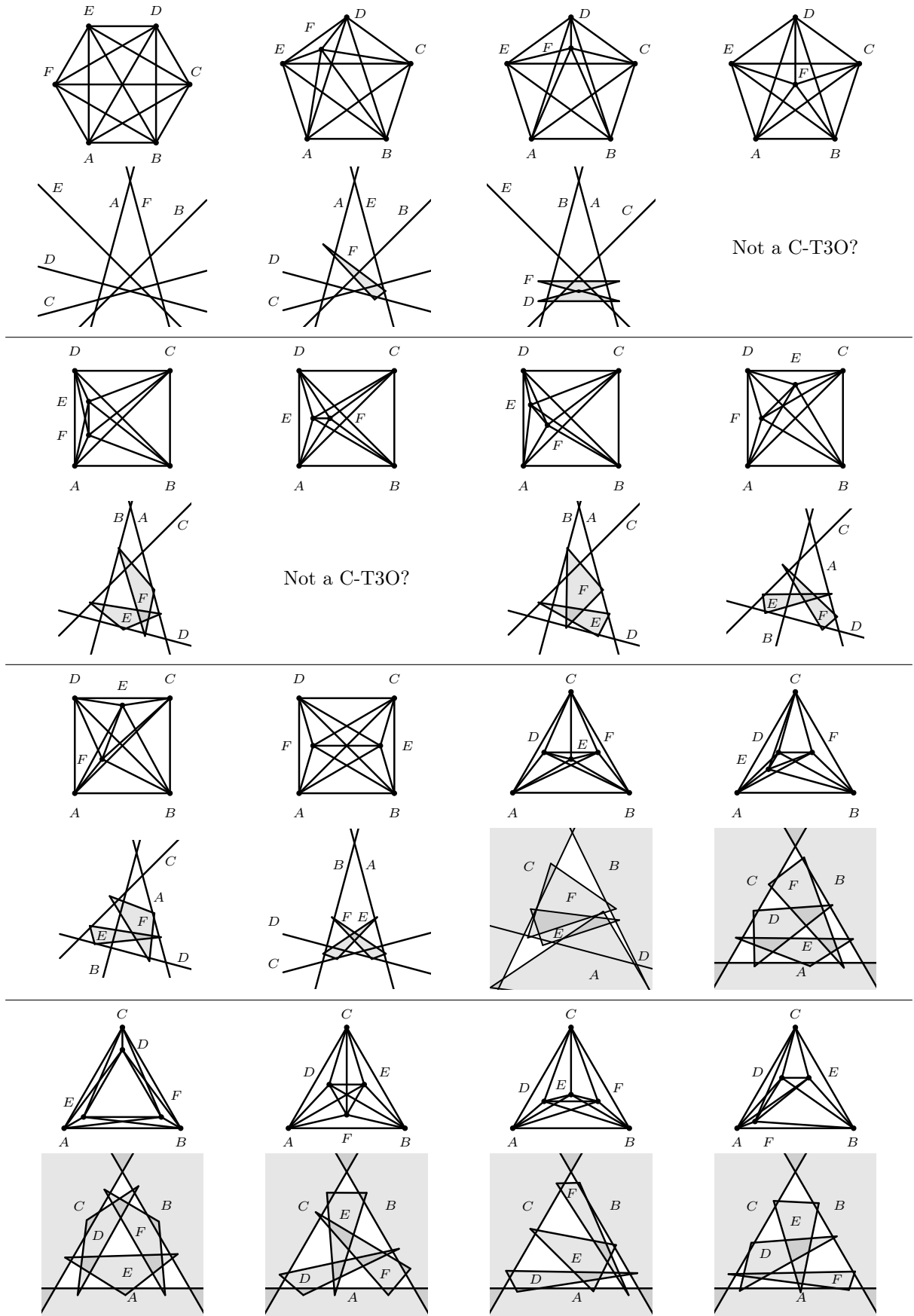

        	\begin{center}
			\begin{minipage}{.24\textwidth}
				\centering
				\scalebox{\kicsiny}{\input{Figures/p1}}
			\end{minipage}
			\begin{minipage}{.24\textwidth}
				\centering
				\scalebox{\kicsiny}{\input{Figures/p2}}
			\end{minipage}
			\begin{minipage}{.24\textwidth}
				\centering
				\scalebox{\kicsiny}{\input{Figures/p3}}
			\end{minipage}
			\begin{minipage}{.24\textwidth}
				\centering
				\scalebox{\kicsiny}{\input{Figures/p4}}
			\end{minipage}
		\end{center}
		\vspace{-20pt}
		\begin{center}
			\begin{minipage}{.24\textwidth}
				\centering
				\scalebox{\kicsiny}{\input{Figures/h1}}
			\end{minipage}
			\begin{minipage}{.24\textwidth}
				\centering
				\scalebox{\kicsiny}{\input{Figures/h2}}
			\end{minipage}
			\begin{minipage}{.24\textwidth}
				\centering
				\scalebox{\kicsiny}{\input{Figures/h3}}
			\end{minipage}
			\begin{minipage}{.24\textwidth}
				\centering
			Not a \CTO?
			\end{minipage}
		\end{center}
	\vspace{-5pt}
		\hrule
	\vspace{-10pt}
        \begin{center}
			\begin{minipage}{.24\textwidth}
				\centering
				\scalebox{\kicsiny}{\input{Figures/p5}}
			\end{minipage}
			\begin{minipage}{.24\textwidth}
				\centering
				\scalebox{\kicsiny}{\input{Figures/p6}}
			\end{minipage}
			\begin{minipage}{.24\textwidth}
				\centering
				\scalebox{\kicsiny}{\input{Figures/p7}}
			\end{minipage}
			\begin{minipage}{.24\textwidth}
				\centering
				\scalebox{\kicsiny}{\input{Figures/p8}}
			\end{minipage}
		\end{center}
		\vspace{-20pt}
		\begin{center}
			\begin{minipage}{.24\textwidth}
				\centering
				\scalebox{\kicsiny}{\input{Figures/h5}}
			\end{minipage}
			\begin{minipage}{.24\textwidth}
				\centering
				Not a \CTO?
			\end{minipage}
			\begin{minipage}{.24\textwidth}
				\centering
				\scalebox{\kicsiny}{\input{Figures/h7}}
			\end{minipage}
			\begin{minipage}{.24\textwidth}
				\centering
				\scalebox{\kicsiny}{\input{Figures/h8}}
			\end{minipage}
		\end{center}
	\vspace{-5pt}
		\hrule
	\vspace{-10pt}
		\begin{center}
			\begin{minipage}{.24\textwidth}
				\centering
				\scalebox{\kicsiny}{\input{Figures/p9}}
			\end{minipage}
			\begin{minipage}{.24\textwidth}
				\centering
				\scalebox{\kicsiny}{\input{Figures/p10}}
			\end{minipage}
			\begin{minipage}{.24\textwidth}
				\centering
				\scalebox{\kicsiny}{\input{Figures/p11}}
			\end{minipage}
			\begin{minipage}{.24\textwidth}
				\centering
				\scalebox{\kicsiny}{\input{Figures/p12}}
			\end{minipage}
		\end{center}
		\vspace{-20pt}
		\begin{center}
			\begin{minipage}{.24\textwidth}
				\centering
				\scalebox{\kicsiny}{\input{Figures/h9}}
			\end{minipage}
			\begin{minipage}{.24\textwidth}
				\centering
				\scalebox{\kicsiny}{\input{Figures/h10}}
			\end{minipage}
			\begin{minipage}{.24\textwidth}
				\centering
				\scalebox{\kicsiny}{\input{Figures/h11}}
			\end{minipage}
			\begin{minipage}{.24\textwidth}
				\centering
				\scalebox{\kicsiny}{\input{Figures/h12}}
			\end{minipage}
		\end{center}
	\vspace{-5pt}
		\hrule
	\vspace{-10pt}
		\begin{center}
			\begin{minipage}{.24\textwidth}
				\centering
				\scalebox{\kicsiny}{\input{Figures/p13}}
			\end{minipage}
			\begin{minipage}{.24\textwidth}
				\centering
				\scalebox{\kicsiny}{\input{Figures/p14}}
			\end{minipage}
			\begin{minipage}{.24\textwidth}
				\centering
				\scalebox{\kicsiny}{\input{Figures/p15}}
			\end{minipage}
			\begin{minipage}{.24\textwidth}
				\centering
				\scalebox{\kicsiny}{\input{Figures/p16}}
			\end{minipage}
		\end{center}
		\vspace{-25pt}
		\begin{center}
			\begin{minipage}{.24\textwidth}
				\centering
				\scalebox{\kicsiny}{\input{Figures/h13}}
			\end{minipage}
			\begin{minipage}{.24\textwidth}
				\centering
				\scalebox{\kicsiny}{\input{Figures/h14}}
			\end{minipage}
			\begin{minipage}{.24\textwidth}
				\centering
				\scalebox{\kicsiny}{\input{Figures/h15}}
			\end{minipage}
			\begin{minipage}{.24\textwidth}
				\centering
				\scalebox{\kicsiny}{\input{Figures/h16}}
			\end{minipage}
		\end{center}
		\vspace*{-3pt}
		\caption{\CTO's representing $6$-point order types.}
		\label{fig:ordertypes}
	\end{figure}
	
We have managed to realize 14 out of the 16 \pTO's as \CTO's, while we conjecture that the other two cannot be realized. The list of these realizations can be found in Figure \ref{fig:ordertypes}.
A difficult-to-realize example is the one depicted in Figure \ref{fig:hatponthathalmaz} left; see the realization in Figure \ref{fig:hatponthathalmaz} right.
The difficulty is due to the following statement.

\begin{prop}\label{prop:noreg}
    If the containments $D,E,F\in conv(ABC)$ are all regular, then there is no realization of the order type of the point set in Figure \ref{fig:hatponthathalmaz} left by convex sets.
\end{prop}

The remainder of this section contains the proof of this proposition.

The main idea of the proof is the following. Suppose we have $\o(ABC)=1$ and all three of the containments ${D, E, F\in conv(ABC)}$ are regular. Then we can determine the orientation of each triple simply  considering the order in which $D\cap \partial A, E\cap \partial A, F\cap \partial A, D\cap \partial B, E\cap \partial B,  F\cap \partial B, D\cap \partial C, E\cap \partial C,  F\cap \partial C$ follow each other on the boundary of $\hollow(ABC)$. For example, if $D\cap \partial A$ comes before $E\cap \partial A$ as we go around $\partial \hollow(ABC)$ counterclockwise, then $\o(AED)=1$, otherwise $\o(AED)=-1$. The only nontrivial case is the orientation of $DEF$:

\begin{claim}
	If in a holey family all three of the containments $D,E,F\in conv(ABC)$ are regular,
	then the order in which $D\cap \partial A, E\cap \partial A, F\cap \partial A, D\cap \partial B, E\cap \partial B,  F\cap \partial B, D\cap \partial C, E\cap \partial C,  F\cap \partial C$
	follow each other around the boundary of $\hollow(ABC)$ determines the orientation of $DEF$.
\end{claim}

\begin{proof}
	Without loss of generality, we can assume $\o(ABC)=1$.
	If $D$, $E$ and $F$ intersect $A$, in an order $RST$, $B$ in an order $UVW$ and $C$ in an order $XYZ$ going around $\hollow(ABC)$ in a positive order (where $R$ and the other just introduced letters are variables that stand for $D,E$ and $F$, i.e., $\lbrace R,S,T\rbrace=\lbrace U,V,W\rbrace=\lbrace X,Y,Z\rbrace=\lbrace D,E,F\rbrace$), denote this configuration by $\left(RST\vert UVW\vert XYZ\right)$.
	
	If $R$, $U$ and $X$ are pairwise different, the cyclic order $RSTUVWXYZ$ contains a $DEFDEF$ or a $DFEDFE$, since one of $S$ and $T$ equals to $X$, one of $V$ and $W$ equals to $R$ and one of $Y$ and $Z$ equals to $U$. Choosing these along with $X$, $Y$ and $Z$, we get the desired order.
	
	If two of $R$, $U$ and $X$ are identical, say, $R=U$, we also get one of the above two cyclic orders: $R$, $S$, $T$, $U$ suffice along with one of $V$ or $W$ and with one of $X$, $Y$ or $Z$.
	
	Take a point from each of these six boundary parts and  connect the points that correspond to the same sets from $D,E,F$, e.g., connect a point from $D\cap \partial A$ to a point from $D\cap \partial B$.
	The resulting three segments pairwise intersect each other because of the regularity of the containment.
	The orientation of these segments is determined by the order of their endpoints, and this also determines the orientation of $D$, $E$ and $F$ because of Corollary \ref{cor:subset}.
\end{proof}

It is not hard to enumerate the combinatorially different possible options for the order in which $D\cap \partial A, E\cap \partial A, F\cap \partial A, D\cap \partial B, E\cap \partial B,  F\cap \partial B, D\cap \partial C, E\cap \partial C,  F\cap \partial C$ can follow each other; these 10 cases can be read off from Figure \ref{fig:tizeset}.

\begin{claim}
	Out of the ten combinatorially different configurations described above, seven are realizable, while the other three are not.
\end{claim}

\begin{proof}
	The first seven drawings in Figure \ref{fig:tizeset} give realizations of the respective configurations.
	For the 8th and 9th configurations, there are two triples of segments (a red and a blue) such that the red segments imply $\o(DEF)=1$ using Corollary \ref{cor:subset}, while the blue segments imply $\o(DEF)=-1$, which gives a contradiction.
	Finally, in the last configuration $(RST|UVW|XYZ)=(DEF|DEF|DEF)$, we will show that we always have a triple intersection.
	For all pairs $M\in\lbrace D,E,F\rbrace$ and $N\in\lbrace A,B,C\rbrace$, choose an arbitrary point $P\left(M,N\right)\in M\cap\partial{N}$.
	We can assume that $M=conv\left(P\left(M,A\right)P\left(M,B\right)P\left(M,C\right)\right)$, as this can only reduce the size of $D\cap E\cap F$. Now as the segment $P(D,A)P(D,B)$ does not cross neither $P(F,C)P(F,B)$, nor $P(E,C)P(E,B)$, the only connected component of $E\setminus F$ it can cross is the one containing $P(E,A)$. Similarly, $P(D,A)P(D,C)$ crosses into the connected component of $E\setminus F$ containing $P(E,C)$. And since $D$ does not intersect $E\cap F$, $D\cap E$ has more than one separate connected components, a contradiction.
\end{proof}

	\begin{figure}
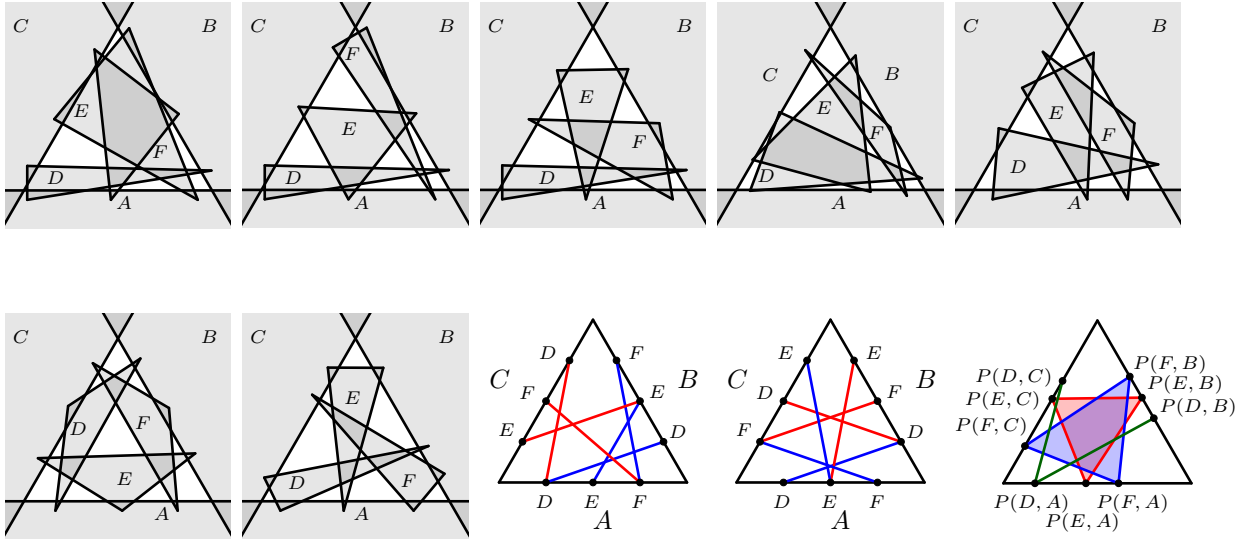

		\begin{center}
			\begin{minipage}{.19\textwidth}
				\centering
				\vspace{-14pt}
				\input{Figures/hathalmaz2}
			\end{minipage}
			\begin{minipage}{.19\textwidth}
				\centering
				\input{Figures/hathalmaz3}
			\end{minipage}
			\begin{minipage}{.19\textwidth}
				\centering
				\input{Figures/hathalmaz4}
			\end{minipage}
			\begin{minipage}{.19\textwidth}
				\centering
				\input{Figures/hathalmaz5}
			\end{minipage}
			\begin{minipage}{.19\textwidth}
				\centering
				\input{Figures/hathalmaz6}
			\end{minipage}
		\end{center}
		
		\begin{center}
			\begin{minipage}{.19\textwidth}
				\centering
				\input{Figures/hathalmaz7}
			\end{minipage}
			\begin{minipage}{.19\textwidth}
				\centering
				\input{Figures/hathalmaz8}
			\end{minipage}
			\begin{minipage}{.19\textwidth}
				\centering
				\input{Figures/hathalmaz9}
			\end{minipage}
			\begin{minipage}{.19\textwidth}
				\centering
				\input{Figures/hathalmaz10}
			\end{minipage}
			\begin{minipage}{.19\textwidth}
				\centering
				\input{Figures/hathalmaz1}
			\end{minipage}
		\end{center}
		\caption{The 10 combinatorially different cases of boundary order.}
		\label{fig:tizeset}
	\end{figure}

\begin{proof}[Proof of Proposition \ref{prop:noreg}.]
 The non-realizablity of Figure \ref{fig:hatponthathalmaz} with regular containments simply follows from checking all possible orderings of $D\cap \partial A, E\cap \partial A, F\cap \partial A, D\cap \partial B, E\cap \partial B,  F\cap \partial B, D\cap \partial C, E\cap \partial C,  F\cap \partial C$. We find that it must correspond to the second figure in Figure \ref{fig:tizeset}, but the orientation of $DEF$ is not correct, hence there is no realization with regular containments. 
\end{proof}

\section{A \pPO that is not a \CPO}\label{app:3int}

The goal of this section is to prove the following theorem.

\begin{thm}\label{thm:5point}
	There is a \pPO that is not a \CPO, i.e., a partial 3-order realizable by five points (not in general position) that is not realizable by intersecting planar convex sets.
\end{thm}

The configuration consists of the four points in convex position and the intersection point of the diagonals of their convex hull (see Figure \ref{fig:squarecenter}).

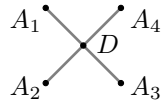
\begin{figure}[!ht]
	\centering
	\begin{tikzpicture}[line cap=round,line join=round,>=triangle 45,x=1.0cm,y=1.0cm]
	\clip(3.4313913628190748,0.6958986342155006) rectangle (5.746334207928435,2.15545872957324);
	\draw [line width=1.pt,color=gray] (4,1)-- (5,2);
	\draw [line width=1.pt,color=gray] (4,2)-- (5,1);
	\draw (3.4,2.1) node[anchor=north west] {$A_1$};
	\draw (5.0,2.1) node[anchor=north west] {$A_4$};
	\draw (3.4,1.2) node[anchor=north west] {$A_2$};
	\draw (5.0,1.2) node[anchor=north west] {$A_3$};
	\draw (4.55,1.77) node[anchor=north west] {$D$};
	\draw [fill=black] (4.,1.) circle (1.0pt);
	\draw [fill=black] (5.,1.) circle (1.0pt);
	\draw [fill=black] (5.,2.) circle (1.0pt);
	\draw [fill=black] (4.,2.) circle (1.0pt);
	\draw [fill=black] (4.5,1.5) circle (1.0pt);
	\end{tikzpicture}
	\caption{A \pPO that is not a \CPO.}
	\label{fig:squarecenter}
\end{figure}

 \begin{proof}
 We start with two simple observations.  
 
 \begin{obs}\label{obs:vertex}
If $\o(ABC)=1$, $x\notin B$ and $x\in A\cap C$, then we can find the vertex of $\hollow(ABC)$ which is not in $C$ as follows. Take the connected component of $A\setminus B$ that contains $x$ and consider its boundary. It consists of two continuous parts, one belonging to $A$ and one belonging to $B$. The counterclockwise last point of the $B$ part will be the desired vertex of $\hollow(ABC)$. Similarly, if $\o(ABC)=-1$, then we find a vertex going clockwise. 
\end{obs}

\begin{figure}[!h]
    \centering
    \begin{tikzpicture}[line cap=round,line join=round,>=triangle 45,x=1.0cm,y=1.0cm]
\clip(-0.07602749121724942,1.6450743804931767) rectangle (3.762330664844489,5.292785608273691);
\draw [rotate around={-89.63272439266275:(2.0706736018761926,3.4667351285751637)},line width=1.pt] (2.0706736018761926,3.4667351285751637) ellipse (1.6676648912567558cm and 0.5894117317549916cm);
\draw [rotate around={0.:(1.9872902047812546,4.025419590437493)},line width=1.pt] (1.9872902047812546,4.025419590437493) ellipse (1.6102465949000657cm and 0.585571598002575cm);
\draw [line width=1.pt] (2.3388336003447656,2.293273936649295)-- (1.0214745148399194,1.959313838741235);
\draw [line width=1.pt] (0.05107046097022513,4.390390147742623)-- (1.0214745148399194,1.959313838741235);
\draw [line width=1.pt] (2.3388336003447656,2.293273936649295)-- (0.05107046097022513,4.390390147742623);
\draw [line width=1.pt,dash pattern=on 3pt off 3pt] (1.4812364886006877,3.4695169544879247)-- (1.8304417915948679,2.5220502505867493);
\begin{scriptsize}

\draw (1.9,5.0) node[anchor=north west] {$A$};
\draw (2.9,4.3) node[anchor=north west] {$B$};
\draw (1.75,2.6) node[anchor=north west] {$x$};
\draw (0.9,2.9) node[anchor=north west] {$C$};
\end{scriptsize}
\end{tikzpicture}\caption{Finding a vertex of a hollow.}
    \label{fig:vertex}
\end{figure}
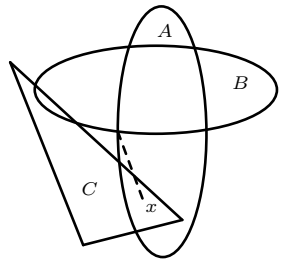

\begin{obs}\label{obs:order}
 Suppose $A$ and $B$ are intersecting convex sets. Let the connected components of $A\setminus B$ be $A^1,\dots, A^k$ and let the connected components of $B\setminus A$ be $B^1,\dots, B^k$, numbered such that $A^i$ is followed by $B^i$ and then by $A^{i+1}$ around $A\cap B$. If $x\in A\cap B$, $a_i\in A^i$ and $b_i\in B^i$, then the order of the $xa_i$ and $xb_i$ rays around $x$ is the same as the order of the $A^i$-s and $B^i$-s around $A\cap B$.
\end{obs}
 We will arrive at a contradiction by showing that $A_1\cap D$ and $A_3\cap D$ are disjoint. Consider first how $A_2$ and $A_4$ lie in the plane. Since they are convex, the connected components of $A_2\setminus A_4$ and $A_4\setminus A_2$ alternate around $A_2\cap A_4$. Let us call these components $A_2^1, \dots,  A_2^k$ for $A_2$ and  $A_4^1, \dots, A_4^k$ for $A_4$ so that $A_2^1$ is followed by $A_4^1$ counterclockwise (see Figure \ref{fig:5pointproof}). Since $\o(A_1A_2A_4)=1$, there is an $i_1$ such that $A_1$ intersects $A_4^{i_1}$ and $A_2^{i_1+1}$ but no other $A_2^j$ or $A_4^j$.  Since $\o(A_3A_2A_4)=-1$, there is an $i_3$ such that $A_3$ intersects $A_2^{i_3}$ and $A_4^{i_3}$ but no other $A_2^j$ or $A_4^j$. This implies that $A_1$ and $A_3$ connect different pairs of the $A_2^j$ or $A_4^j$ sets (though one member of the two pairs might be the same). 

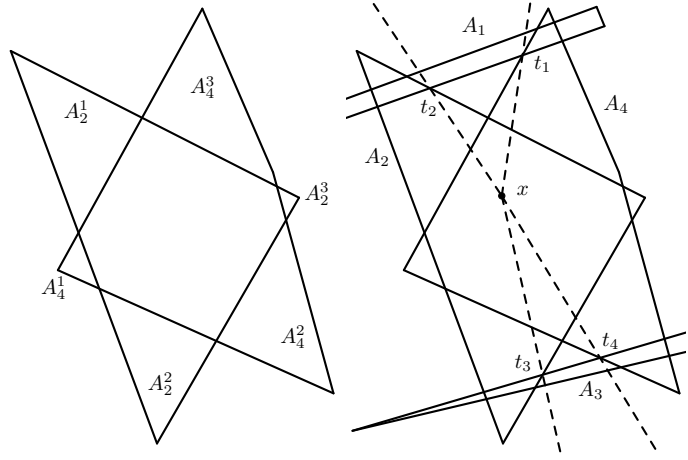
\begin{figure}[!h]
    \centering
   \scalebox{0.8}{
    \begin{tikzpicture}[line cap=round,line join=round,>=triangle 45,x=1.0cm,y=1.0cm]
\clip(2.974397834183812,-0.7589292007818002) rectangle (8.606636592342204,6.729204850179495);
\draw [line width=1.pt] (3.143898791403217,5.925014783532181)-- (5.574386801921873,-0.6031554926542864);
\draw [line width=1.pt] (5.574386801921873,-0.6031554926542864)-- (7.934571440235445,3.4844834341270867);
\draw [line width=1.pt] (7.934571440235445,3.4844834341270867)-- (3.143898791403217,5.925014783532181);
\draw [line width=1.pt] (3.927279224545594,2.2792827677542005)-- (6.327637218404928,6.628048505583032);
\draw [line width=1.pt] (6.327637218404928,6.628048505583032)-- (7.5027078681184936,3.906303667357597);
\draw [line width=1.pt] (7.5027078681184936,3.906303667357597)-- (8.507041756762566,0.23044163492029343);
\draw [line width=1.pt] (8.507041756762566,0.23044163492029343)-- (3.927279224545594,2.2792827677542005);
\draw (3.9,5.228350395999753) node[anchor=north west] {$A_2^1$};
\draw (3.5230973120559783,2.3) node[anchor=north west] {$A_4^1$};
\draw (5.3,0.7) node[anchor=north west] {$A_2^2$};
\draw (7.5,1.4842833705191045) node[anchor=north west] {$A_4^2$};
\draw (7.9,3.856601701319343) node[anchor=north west] {$A_2^3$};
\draw (6.0,5.631805894435168) node[anchor=north west] {$A_4^3$};
\end{tikzpicture}
    \begin{tikzpicture}[line cap=round,line join=round,>=triangle 45,x=1.0cm,y=1.0cm]
\clip(2.974397834183812,-0.7589292007818002) rectangle (8.606636592342204,6.729204850179495);
\draw [line width=1.pt] (3.143898791403217,5.925014783532181)-- (5.574386801921873,-0.6031554926542864);
\draw [line width=1.pt] (5.574386801921873,-0.6031554926542864)-- (7.934571440235445,3.4844834341270867);
\draw [line width=1.pt] (7.934571440235445,3.4844834341270867)-- (3.143898791403217,5.925014783532181);
\draw [line width=1.pt] (3.927279224545594,2.2792827677542005)-- (6.327637218404928,6.628048505583032);
\draw [line width=1.pt] (6.327637218404928,6.628048505583032)-- (7.5027078681184936,3.906303667357597);
\draw [line width=1.pt] (7.5027078681184936,3.906303667357597)-- (8.507041756762566,0.23044163492029343);
\draw [line width=1.pt] (8.507041756762566,0.23044163492029343)-- (3.927279224545594,2.2792827677542005);
\draw (3.15,4.421439399128923) node[anchor=north west] {$A_2$};
\draw (7.121920358099882,5.309041495686835) node[anchor=north west] {$A_4$};
\draw [line width=1.pt] (2.6414609280147197,5.0099247488156635)-- (7.129812755939393,6.657886763247708);
\draw [line width=1.pt] (7.129812755939393,6.657886763247708)-- (7.261667734843916,6.336791677876251);
\draw [line width=1.pt] (7.261667734843916,6.336791677876251)-- (2.55773865072346,4.66495062090781);
\draw [line width=1.pt] (2.55773865072346,4.66495062090781)-- (2.6414609280147197,5.0099247488156635);
\draw (4.765740247237058,6.6) node[anchor=north west] {$A_1$};
\draw (5.68,3.8) node[anchor=north west] {$x$};
\draw [line width=1.pt,dash pattern=on 4pt off 4pt,domain=5.554300124148992:10.60777586458187] plot(\x,{(-11.772671360348046--2.338111682172511*\x)/0.3453872419299264});
\draw [line width=1.pt,dash pattern=on 4pt off 4pt,domain=2.022242857876233:5.554300124148992] plot(\x,{(-14.144632996980802--1.7908825240626056*\x)/-1.1943088564417792});
\draw [line width=1.pt] (3.073595419198132,-0.3922453760390312)-- (9.069468734403248,1.3753822679745356);
\draw [line width=1.pt] (9.069468734403248,1.3753822679745356)-- (9.270335512132062,1.074082101381314);
\draw [line width=1.pt] (9.270335512132062,1.074082101381314)-- (3.073595419198132,-0.3922453760390312);
\draw [line width=1.pt,dash pattern=on 4pt off 4pt,domain=5.554300124148992:10.60777586458187] plot(\x,{(--18.914044889960035-2.975101746637195*\x)/0.6798576636422231});
\draw [line width=1.pt,dash pattern=on 4pt off 4pt,domain=5.554300124148992:10.60777586458187] plot(\x,{(--20.704128523880787-2.693970416105672*\x)/1.6334679155019938});
\draw (6.70232663972705,0.5482666141489424) node[anchor=north west] {$A_3$};
\draw (5.97,5.97) node[anchor=north west] {$t_1$};
\draw (4.08,5.24) node[anchor=north west] {$t_2$};
\draw (7.1,1.3551776110197717) node[anchor=north west] {$t_4$};
\draw (5.65,1.00) node[anchor=north west] {$t_3$};
\begin{scriptsize}
\draw [fill=black] (5.554300124148992,3.514613450786409) circle (1.5pt);
\end{scriptsize}
\end{tikzpicture}
        }
    \caption{Proof of Theorem \ref{thm:5point}.}
    \label{fig:5pointproof}
\end{figure}

Fix a point $x$ such that $x\in A_2\cap A_4\cap D$. We apply Observation \ref{obs:vertex} four times. First, for $A_4,A_1,D$ and $x$, which gives us the vertex $t_1$ of $\hollow(A_1,A_4,D)$ 
that is not contained in $D$. Hence, the ray $xt_1$ first runs in $D$, then it leaves it, then it enters $A_1$ at $t_1$. This means that $D$ cannot intersect $A_1$ on both sides of $xt_1$, because then it would contain $t_1$. 

Similarly, define $t_2$ for the sets $A_2,A_1,D$. Since $D$ is convex and it intersects $A_1$ on only one side of the ray $xt_2$, the only possibility is that $A_1\cap D$ lies in the counterclockwise cone bounded first by the ray $xt_1$ and then ended at $xt_2$. 

Apply the same reasoning for $A_2,A_3,D$ and $A_4,A_3,D$, giving us $t_3$ and $t_4$ (see Figure \ref{fig:5pointproof}). $A_3\cap D$ lies in the counterclockwise cone bounded first by the ray $xt_3$ and then ended at $xt_4$.  We show that the two cones are disjoint apart from $x$. Therefore, we have $D\cap A_1\cap A_3=\emptyset$, a contradiction. 

To see that the cones are disjoint, note that $A_1$ and $A_3$ connect different pairs of the $A_2^j$ or $A_4^j$ sets, but they connect neighbouring ones. Consider first the case when $t_1, t_2, t_3$ and $t_4$ are in different $A_i^j$-s. Using Observation \ref{obs:order} we can see that $xt_1, xt_2, xt_3, xt_4$ comes in this order around $x$ as $t_1$ is surely followed by $t_2$ and $t_3$ is followed by $t_4$.

Hence, the only two cases that we have to consider is when $t_1$ and $t_4$ are in the same $A_4^j$, or when $t_2$ and $t_3$ are in the same $A_2^j$. By symmetry, it is enough to consider the first case, so suppose $t_1$ and $t_4$ lies in $A_4^j$. Pick $q_1\in A_2^{j}\cap A_3$ and $q_2\in A_2^{j+1}\cap A_1$. Since $t_1$ was the last point counterclockwise, $q_2t_1$ does not enter the interior of $A_4$. Since $t_4$ was the last point clockwise, $q_1t_4$ does not enter the interior of $A_4$. That is, $t_1$ can be seen on the boundary of $A_4^{j}$ from $A_2^{j+1}$ and $t_4$ can be seen from $A_2^{j}$. But this implies that they are in the right order, $t_4$ comes first on the boundary of $A_4^j$ going counterclockwise.
\end{proof}

\section{Towards $(p,q)$-theorems}\label{app:43}

Holmsen and Lee \cite{H,HL} (see also \cite{AKMM}) have shown that practically the implication of the $d$-dimensional colorful Helly theorem \cite{colHelly} alone is sufficient to obtain $(p,d+1)$-theorems \cite{AK} for any set system, i.e., to show that if from any $p$ sets there are $d+1$ with a common point, then there are $C(p,d+1)$ points that stab every set.
The constant their general method obtains, however, is quite weak.
Our motivation in this section is to try to improve this constant using properties of \CPO's.
We focus on the simplest open case for planar convex sets, (4,3), so we assume that from any 4 sets at least 3 intersect.
For this special case, better bounds are known but the current best upper bound, $9$ \cite{McG} is far from the conjectured $3$ \cite{KGyT}.

A possible way to characterize the intersections occurring in a family of sets is to represent the sets by the vertices of a hypergraph in which those tuples that do not intersect form a hyperedge. Note that the complement of this hypergraph (where the intersecting tuples form a hyperedge) is possibly more natural for some purposes (for example, the nerves mentioned in Section \ref{sec:introduction} are analogous to this version), but for characterizing \CPO's, we found this version better.

For convex sets in $\mathbb{R}^d$, only the edges on at most $d+1$ vertices contain information.
Indeed, by Helly's theorem, a collection $S$ of more than $d+1$ vertices forms a hyperedge if and only if at least one $(d+1)$-tuple from $S$ forms a hyperedge. We consider the case when the planar convex sets are also pairwise intersecting, i.e., no pairs of vertices form a hyperedge---at the cost of $1$ extra stabbing point the original $(4,3)$ problem reduces to this case, see \cite{KGyT}---so all the information about the intersection structure of the family is described by a $3$-uniform hypergraph.

We defined orientations in order to get extra information on the intersection structure of such families: in the specific case, we can indicate this extra information by defining an orientation on the ordered triples forming a hyperedge that coincides with the corresponding orientation within the \CPO. But also in general, we may regard any P$k$O as a $k$-uniform hypergraph in which hyperedges are oriented according to the orientation of the particular $k$-tuple. We will sometimes use this terminology and also refer to a $k$-tuple in a P$k$O whose orientation is $0$ as a non-edge, and to a subset of vertices in which all $k$-tuples form a non-edge as an independent set. 
A $k$-uniform hypergraph in which all ordered $k$-tuples are assigned a value of $\pm1$ in a way that this sign assignment satisfies the orientation property (but not necessarily the interiority condition), will be called an \emph{oriented $k$-uniform hypergraph}.

\smallskip

First, we will summarize a few results about the structure of oriented hypergraphs for lower-dimensional or otherwise simpler structures, before considering the \CPO case.

Recall that a P2O is a poset, thus it is always a comparability graph with some orientation.
Moreover, the edges of any comparability graph can be directed (oriented) to get a P2O.

A C-P2O is defined by one-dimensional convex sets, i.e., by intervals on a line: a pair of intervals has orientation $0$ if they intersect and otherwise their orientation depends on which side of each other they lie on.
It is easy to see that this way we indeed obtain a poset, i.e., a P2O. The graphs representing the intersection structure of intervals on the line are called interval graphs, which are known to be equivalent to the class of induced $C_4$-free co-comparability graphs \cite{GH}. In other words, exactly those graphs are interval graphs, for which there is an orientation of the non-edges of the interval graph that gives a poset, while the interval graph is required to satisfy the colorful Helly theorem \cite{colHelly}, which says in one dimension that there is no induced $K_{2,2}$ in the interval graph. Thus, a graph (2-uniform hypergraph) can be directed (oriented) to get a C-2PO if and only if it is a comparability graph that does not contain two disjoint edges as an induced subgraph.

Mirsky's theorem \cite{Mirsky} (the dual of Dilworth's theorem \cite{Dilworth} for posets) states that a poset with maximal chain length $p-1$ can be partitioned into $p-1$ antichains, which in our terminology means that in a comparability graph in which the size of the maximal clique has size $p-1$, there are at most $p-1$ independent sets.
In the interval language this means that if in a family of intervals on a line there are at most $p-1$ pairwise non-intersecting intervals, then the family can be divided into $p-1$ subfamilies whose members are pairwise intersecting and thus, because of Helly's theorem, the family can be stabbed by $p-1$ points.
Therefore, Mirsky's theorem gives us a $(p,2)$-theorem.

The situation is different for P3O's: all $3$-uniform hypergraphs can be oriented to obtain a P3O. 
Just take an arbitrary ordering $v_1,\ldots,v_n$ of the vertices and orient the triples forming a hyperedge in a way that the increasing triples according to this ordering get an orientation of $+1$, i.e., $\o(v_iv_jv_k)=1$ if and only if $v_iv_jv_k$ is a hyperedge and the sign of the permutation $(i,j,k)$ is positive.
For such an orientation, the premise of the interiority condition cannot hold at all, and thus the condition is vacuously true, so the orientation is indeed a P3O.\\

In the rest of this section, we study \CPO's.

Similarly to the one-dimensional case, on the plane, the colorful Helly theorem provides a forbidden induced subhypergraph family for the 3-uniform hypergraph of triple non-intersections of convex sets.
It is forbidden to have an induced subhypergraph whose vertices are partitioned into 3 groups, $V_1, V_2, V_3$, such that for every $v_1\in V_1, v_2\in V_2, v_3\in V_3$, $v_1v_2v_3$ is not a hyperedge, but $\binom{V_i}3$ contains a hyperedge for all $i$.
Forbidding these, however, cannot be a sufficient characterization for \CPO, as any 3-uniform hypergraph can be oriented to obtain a \PO as we have already seen.

Below we prove some more rules that every \CPO needs to satisfy.
For example, for four pairwise intersecting convex sets it can be easily checked that the following rule is enough to exclude those \PO's that satisfy the interiority condition but cannot be represented by convex sets.

\begin{lem}\label{lem:doubleintersection}
	If $\o(ABC)=\o(ABD)=0$ but $\o(ACD)\ne 0$ and
	$\o(BCD)\ne 0$, then $\o(ACD)=\o(BCD)$ for any four pairwise intersecting planar convex sets.
\end{lem}

In fact, Lemma \ref{lem:doubleintersection} can be strengthened as follows.
(We get back Lemma \ref{lem:doubleintersection} when $\mathcal{A}=\lbrace A,D\rbrace$, $\mathcal{B}=\lbrace B\rbrace$, $\mathcal{C}=\lbrace C\rbrace$.)

\begin{lem}\label{lem:doubleintersection3}
    If $\mathcal A$, $\mathcal B$, $\mathcal C$ are three families of planar convex sets each such that  $\mathcal A\cup\mathcal B\cup\mathcal C$ is pairwise intersecting and any two-colored triple is intersecting (i.e., $\o(A_1A_2B)=0$ for all $A_1,A_2\in \mathcal A$, $B\in \mathcal B$, and similarly for the other cases of picking two sets from one family, and a third set from another), then the orientation of any non-intersecting colorful triple is the same (i.e., $\o(ABC)\ge 0$ or $\o(ABC)\le 0$ for all $A\in \mathcal A$, $B\in \mathcal B$, $C\in \mathcal C$).
\end{lem}

\begin{proof}
    To prove the statement, take two non-intersecting colorful triples, i.e., $A,A'\in\mathcal{A}$, $B,B'\in\mathcal{B}$ and $C,C'\in\mathcal{C}$ for which $\o(ABC)\neq0\neq\o(A'B'C')$.
    Because the two-colored triples are intersecting, there exist points $x\in B\cap B'\cap C\cap C'$, $y\in A\cap A'\cap C\cap C'$, $z\in A\cap A'\cap B\cap B'$ by Helly's theorem (if some of $A$, $A'$, $B$, $B'$, $C$ and $C'$ are the same, we do not even need Helly's theorem).
    By Lemma \ref{lem:hollow}(d) we have the following equality for the orientations: $\o(ABC)=\o(xyz)=\o(A'B'C')$, which is exactly what we wanted to prove.
\end{proof}

Say that an oriented $3$-uniform hypergraph $H$ (as a special case, a \PO) satisfies the $(4,3)$ property if $H$ contains no $K_4^{(3)}$ (with any orientation of the hyperedges).
Lemma \ref{lem:int'} states that a \PO that has the $(4,3)$ property cannot satisfy the premise of the interiority condition. 

\begin{lem}\label{lem:int'}
	$\o(ABD)=\o(BCD)=\o(CAD)=1$ cannot hold for any $A,B,C,D$ elements of a \PO that satisfies the $(4,3)$ property.
\end{lem}

\begin{proof}
	If these equalities all hold, then from Lemma \ref{lem:int} we would have $\o(ABC)=1$, contradicting the (4,3) property.
\end{proof}

A natural question is whether Lemmas \ref{lem:doubleintersection3} and \ref{lem:int'} are sufficient to get a $(4,3)$-theorem for \PO's, i.e., to decompose the vertices of the corresponding oriented hypergraph into a bounded number of independent sets, i.e., to bound its chromatic number?
The following claim answers this question negatively, as it shows that the size of the largest independent set can be about as small as the square root of the number of vertices.

\begin{claim}\label{cl:sqrtncliquehypergraph}
 For all $k\ge2$, there exists a \PO on $k^2$ elements for which the $(4,3)$ property and the property guaranteed by Lemma \ref{lem:doubleintersection3} hold, but its largest independent set has size $2k-1$.
\end{claim}

\begin{proof}
    Throughout the proof, we will use hypergraph terminology.

	Let the set of vertices be $\left\lbrace\left(i,j\right)\vert 1\le i\le k, 1\le j\le k, i,j\in\mathbb{Z}\right\rbrace$ and for any vertex $v$, use the notation $v=\left(i_V,j_V\right)$.
 
    For three vertices $X$, $Y$ and $Z$, let $\o(XYZ)=0$ if and only if any of the following three conditions holds:
	
	a) $i_X=i_Y=i_Z$.
	
	b) $i_X\neq i_Y\neq i_Z\neq i_X$.

	c) $i_X<i_Y=i_Z$ for some permutation of $X$, $Y$ and $Z$.
	
	This implies that $\o(XYZ)\ne 0$ if and only if $i_X=i_Y<i_Z$ for some permutation of $X$, $Y$ and $Z$.	
    We define $\o(XYZ)=1$ if $i_X=i_Y<i_Z$ and $j_X<j_Y$, and cyclically extend this to other triples to obtain a partial orientation of all triples.

    The structure we defined is an orientation per definition, thus we only have to prove that it satisfies the interiority condition (meaning that it is a \PO), the $(4,3)$ property and the property described in Lemma \ref{lem:doubleintersection3}. 
	
	First, we will prove that for all quadruples $A=\left(i_A,j_A\right)$, $B=\left(i_B,j_B\right)$, $C=\left(i_C,j_C\right)$, $D=\left(i_D,j_D\right)$, there are at most three hyperedges, proving the (4,3) property, and we will simultaneously show that the interiority condition vacuously holds. We can assume $i_A\le i_B\le i_C\le i_D$ and we partition the possible equalities into the following four cases.
	
	\begin{itemize}
	    \item If $i_A=i_B=i_C=i_D$, $i_A<i_B=i_C=i_D$, or $i_A< i_B<i_C\le i_D$, then $A$, $B$, $C$ and $D$ span zero hyperedges.
	    \item If $i_A<i_B=i_C<i_D$, then there is just one hyperedge, $BCD$.
	    \item If $i_A= i_B<i_C\le i_D$, then there are two hyperedges, $ACD$ and $BCD$.
	    \item If $i_A=i_B=i_C<i_D$, then there are three hyperedges (only $ABC$ is missing), and they do not fulfill the premise of the interiority property.
	\end{itemize}	

    As in the first three cases the vertices span less than three vertices, they also cannot fulfill the premise of the interiority property, so it always holds.
	This finishes the proof that our hypergraph is a \PO that has the $(4,3)$ property.\\

Now we prove that the hypergraph also fulfills the criterion of Lemma \ref{lem:doubleintersection3}.

Assume that $\mathcal{A}$, $\mathcal{B}$ and $\mathcal{C}$ are subfamilies for which any two-colored triple (two vertices from one of $\mathcal{A}$, $\mathcal{B}$ and $\mathcal{C}$ and a third vertex from a different one of $\mathcal{A}$, $\mathcal{B}$ and $\mathcal{C}$) always has orientation $0$. Now take vertices $A,A'\in\mathcal{A}$, $B,B'\in\mathcal{B}$ and $C,C'\in\mathcal{C}$ for which $\o(ABC)\neq0\neq\o(A'B'C')$; we will prove that $\o(ABC)=\o(A'B'C')$. First, we may assume, without loss of generality, that $\min\left(i_A,i_B,i_C\right)\le\min\left(i_{A'},i_{B'},i_{C'}\right)$ and (since $\o(ABC)\neq0$, meaning that exactly two of $i_A$, $i_B$ and $i_C$ are minimal) we may also assume, without loss of generality, that $i_A=i_B=\min\left(i_A,i_B\right)<i_C$ (otherwise we may apply a permutation on $\mathcal{A}$, $\mathcal{B}$ and $\mathcal{C}$). If any of $i_{A'}$ or $i_{B'}$ is larger than $i_A=i_B$, then $\o(ABA')$ or $\o(ABB')$ is non-zero, contradicting the assumption on $\mathcal{A}$, $\mathcal{B}$ and $\mathcal{C}$. Thus, since out of $i_{A'}$, $i_{B'}$ and $i_{C'}$, there are also exactly two minimal elements, $i_A=i_B=i_{A'}=i_{B'}$. But if either $A\neq A'$ or $B\neq B'$ holds, this would mean that, say, $\o(AA'C)\ne 0$, contradicting the assumption on the two-colored triples of $\mathcal{A}$, $\mathcal{B}$ and $\mathcal{C}$. Thus, $A=A'$ and $B=B'$ meaning that $\o(ABC)=\o(A'B'C')$, finishing the statement.
	
	Now suppose for a contradiction that there exists an independent set $\kappa$ of size $2k$. By the pigeonhole principle, there are at least two first coordinates, $i_1<i_2$, which belong to at least two vertices from $\kappa$, otherwise there would be at most $(k-1)\cdot 1+k$ vertices in $\kappa$. 
    Then any two vertices with first coordinate $i_1$ along with a vertex with first coordinate $i_2$ form a hyperedge, contradicting $\kappa$ being an independent set.	
	This proves that the size of the largest independent set is at most $2k-1$; an independent set of this size indeed exists: $\left\lbrace (i,j)\vert i=k\right\rbrace\cup\left\lbrace (i,j)\vert j=k\right\rbrace$.
\end{proof}

\begin{remark}\label{rem:sqrtncliquehypergraph}
The construction used in Claim \ref{cl:sqrtncliquehypergraph} is a \PO, but it is not a \CPO (for $k\ge4$) and not a \pPO (for $k\ge2$). 
\end{remark}

\begin{proof}
We already proved in Claim \ref{cl:sqrtncliquehypergraph} that the construction is a \PO.

Also, it is not a p-P3O for $k\ge2$, because two triples from $(1,1)$, $(1,2)$, $(2,1)$ and $(2,2)$ should be on a line, while the other two triples should not be, a contradiction with collinearity.

It is also not a \CPO: as any family of convex sets with the $(4,3)$ property can be stabbed by at most $9$ points \cite{McG}, at least one of these points intersects at least $k^2/9$ of these sets, which correspond to an independent set of size at least $k^2/9$ in the \CPO, which is a contradiction for $k\ge18$.
\end{proof}

In the rest of this section we study a digraph constructed from a \CPO with the (4,3) property by taking the `trace' of hyperedges with a common vertex.
To be more precise, define $G_D$ to be the directed graph whose vertices are the elements of some \CPO $H$, except for one element $D$, and $AB$ is an edge of $G_D$ if and only if $\o(ABD)=1$ in $H$. It follows from Lemma \ref{lem:int'} that the $(4,3)$ property implies that there are no directed $3$-cycles ($C_3$'s) in $G_D$.
We can even prove that the girth is at least five.

\begin{claim}\label{claim:4cycle}
	There is no directed 4-cycle in $G_D$.
\end{claim}
\begin{proof}
	Suppose that $A_1A_2A_3A_4$ is a directed 4-cycle in $G_D$.
	Then $A_1A_3$ cannot be an edge, or we would have a directed 3-cycle.
	If $A_1A_2A_3$ is not a hyperedge in $H$, that would contradict Lemma \ref{lem:doubleintersection} for $A=A_1, B=A_3, C=A_2, D=D$.
	We can similarly argue for the other indices, which implies that $A_1,A_2,A_3,A_4$ contradict the (4,3) property.
\end{proof}

In light of the above it is natural to conjecture that $G_D$ is always acyclic, however, this is not the case.

\begin{claim}
	$G_D$ can contain a directed 5-cycle.
\end{claim}
\begin{proof}
	Take a pentagon $v_1v_4v_2v_5v_3$ with inscribed circle $D$. The set $A_i$ will be the triangle $v_{i+2}v_iv_{i-2}$ where indices are mod 5 (see Figure \ref{fig:pentagon}).
	
	\begin{figure}[h]
		\centering
		\input{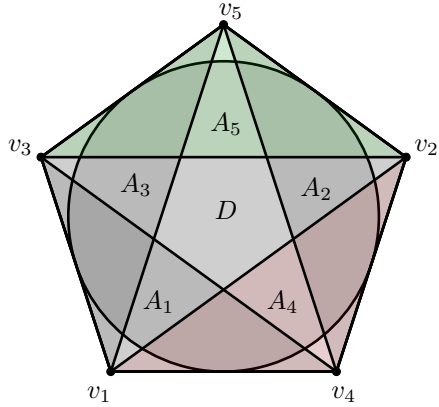}    
		\caption{An example of $G_D$ containing a directed 5-cycle. $D$ is the disk in the center, triangle $A_4$ is shaded in red, and triangle $A_5$ is shaded in green.}\label{fig:pentagon}
	\end{figure}
	
	For these sets $\o(A_iA_{i+1}D)=\o(A_iA_{i+1}A_{i+2})=1$, where indices are to be understood mod 5, and all other triples have orientation $0$ in $H$.
	The set $A_i$ will be the triangle $v_{i+2}v_iv_{i-2}$ where indices are mod 5. Thus, $A_1A_2A_3A_4A_5$ is a $5$-cycle in $G_D$.
\end{proof}

\section{Open problems}\label{sec:discussion}
Our definition of orientation can be generalized to intersecting pseudo-disk arrangements and to $d+1$ convex sets in $\R^d$.
We leave these for future research, just like the following questions left open in this paper.

\begin{prob}
    Are all \CPO's and/or \CTO's extendable by adding one more element?
\end{prob}

\begin{prob}
    Are all 6-point order types \CTO's, or that two that we could not realize in Figure \ref{fig:ordertypes} are not?
\end{prob}

\begin{prob}
    What further properties of \CPO's are needed to obtain efficient $(p,q)$-theorems?
\end{prob}

\subparagraph*{Acknowledgments.}
~\\
We would like to thank N\'ora Frankl and M\'arton Nasz\'odi for discussions during the project and a reviewer for their useful suggestions.


\end{document}